\documentclass[11pt]{article}
\usepackage{geometry}
\usepackage{amsmath, amssymb, amsthm}
\usepackage[sort]{natbib}
\setcitestyle{numbers,square,comma}
\usepackage{lineno}
\usepackage{graphicx}
\usepackage{dcolumn}
\usepackage{multirow}
\usepackage{bm}
\usepackage{color}
\usepackage{tikz}
\usepackage{pgfplots}
\usepackage{mathtools}
\usepackage{array}
\usepackage[hyphens]{url}
\usepackage{hyperref}
\hypersetup{colorlinks=true,
			urlcolor=black,
			linkcolor=black,
			citecolor=blue,
			bookmarksdepth=paragraph,
			pdftitle={Computation of attractor dimension and maximal sums of Lyapunov exponents using polynomial optimization (arXiv version)},
			pdfauthor={Jeremy P. Parker, David Goluskin}}

\usepackage{gnuplottex}
\pgfplotsset{compat=1.18}

\usepackage{thmtools}
\usepackage[noabbrev]{cleveref}

\newtheorem{prop}{Proposition}[section]
\newtheorem{lem}[prop]{Lemma}
\newtheorem{rmk}[prop]{Remark}

\theoremstyle{definition}
\newtheorem{defn}[prop]{Definition}

\crefname{appsec}{appendix}{appendices}
\Crefname{appsec}{Appendix}{Appendices}
\crefname{prop}{proposition}{propositions}
\Crefname{prop}{Proposition}{Propositions}
\crefname{equation}{}{}
\Crefname{equation}{}{}

\DeclareMathOperator*{\diag}{diag}

\newcommand{\ddt}[1]{\frac{\mathrm{d}#1}{\mathrm{d}t}}
\newcommand{\bn}[2]{C_{#1}^{#2}} 
\newcommand{\Mu}{M} % Capital greek letters which are identical to latin ones not actually defined by default

\begin{document}
\title{Computation of attractor dimension and maximal sums of Lyapunov exponents using polynomial optimization}
\author{Jeremy P. Parker$^1$, David Goluskin$^2$}
\date{$^1$Division of Mathematics, University of Dundee, Dundee, DD1 4HN, United Kingdom\\
$^2$Department of Mathematics and Statistics, University of Victoria, Victoria, BC, V8P 5C2, Canada}

\maketitle

\begin{abstract}
Two approaches are presented for computing upper bounds on Lyapunov exponents and their sums, and on the Lyapunov dimension, among all trajectories of a dynamical system governed by ordinary differential equations. The first approach expresses a sum of Lyapunov exponents as a time average in an augmented dynamical system and then applies methods for bounding time averages. This generalizes the method of Oeri \& Goluskin (Nonlinearity 36:5378--5400, 2023) for bounding the single largest Lyapunov exponent. The second approach considers a different augmented dynamical system, where bounds on sums of Lyapunov exponents are implied by stability of certain sets, and such stability is verified using Lyapunov function methods. Both of our approaches also can be adapted to directly compute bounds on Lyapunov dimension, which in turn imply bounds on the fractal dimension of a global attractor. For systems of ordinary differential equations with polynomial right-hand sides, all of our bounding formulations lead to polynomial optimization problems with sum-of-squares constraints. These sum-of-squares problems can be solved computationally for a chosen system to yield numerical bounds, provided the number of variables and degree of polynomials are not prohibitive. Most of our bounding formulations are proved to be sharp under mild assumptions. In the case of the polynomial optimization problems, sharpness means that upper bounds converge to the quantities being bounded as polynomial degrees are raised. Computational examples demonstrate upper bounds that are sharp to several digits, including for a six-dimensional dynamical system where sums of Lyapunov exponents are maximized on periodic orbits.
\end{abstract}

\section{Introduction}
The detection and characterization of chaos is a major computational challenge for many dynamical systems. Chaotic properties are often quantified via the spectrum of Lyapunov exponents (LEs), which can be computed over individual trajectories. When combined with other knowledge about stability or local attractors, LEs can imply, for example, existence of hyperchaos or estimates of attractor dimension. In this paper we formulate upper bounds on sums of LEs, and on the related Lyapunov dimension of a global attractor. These bounds apply to all possible trajectories, yet they do not require knowing any trajectories. Such upper bounds are complementary to typical trajectory-based computations, which give corresponding lower bounds on sums of LEs and Lyapunov dimension.

The LE spectrum quantifies the asymptotic rates at which linearized perturbations in different directions grow or decay along trajectories of a dynamical system. For dynamics in $n$ dimensions there are $n$ exponents, indexed in decreasing order. The sum of the first $k$ exponents for a trajectory gives the exponential rate of expansion or contraction of $k$-dimensional volumes in state space, infinitesimally close to that trajectory. The first exponent governs the stretching of lines, the sum of the first two governs the growth of areas, and so on. In conservative systems, $n$-dimensional volumes in state space are conserved, so the sum of all $n$ LEs is zero. In dissipative systems, the LE sums are negative beyond a certain value of $k$, reflecting the contraction of higher-dimensional volumes. One way of defining a value $k$ at which this contraction begins leads to the Lyapunov dimension, which is non-integer in general, and which provides a natural connection between dynamical stability and the fractal geometry of invariant sets such as strange attractors.

We consider autonomous, continuous-time dynamics governed by an ordinary differential equation (ODE) system,
\begin{equation}
    \ddt{}X(t) = F(X(t)),\qquad X(0)=X_0,
    \label{eq:ODE}
\end{equation}
where each trajectory $X(t)$ remains in the state space $\Omega\subset \mathbb{R}^n$ for $t\ge0$, where $\Omega$ may be all of $\mathbb{R}^n$, a subdomain of $\mathbb{R}^n$ or a lower-dimensional manifold embedded in $\mathbb{R}^n$. Assume that $ F\in C^1(\Omega,\mathbb{R}^n)$, where $C^k(A,B)$ denotes the space of functions mapping $A$ to $B$ that are $k$ times continuously differentiable, and when $B=\mathbb{R}$ we write $C^k(A)$.

The present work builds on existing methods for computing global properties of ODE dynamical systems. These methods rely on finding auxiliary functions $V\in C^1(\Omega)$ subject to certain pointwise inequalities on the state space that,  in turn, imply statements about the dynamics. The best known examples of auxiliary functions are Lyapunov functions, whose inequality conditions imply statements about stability. (Our use of Lyapunov functions to study Lyapunov exponents is not typical, despite both objects bearing the same person's name.) Among the lesser known types of auxiliary functions are those whose inequality conditions imply bounds on infinite-time averages. For each quantity studied in the present work, we formulate two different upper bounds: a ``shifted spectrum'' approach based on Lyapunov function conditions for stability, and a ``sphere projection'' approach based on auxiliary function conditions for bounding time averages. 

The shifted spectrum approach that we present in \cref{sec:shifted} uses conditions on $V$ that are similar to standard Lyapunov function conditions in stability theory. Such pointwise inequality conditions may be, for instance,
\begin{equation}
\label{eq:lyapunovfunction}
    V(X)\geq0,\qquad F(X)\cdot DV(X)\le 0 \qquad \forall\, X\in\Omega,
\end{equation}
where $DV$ denotes the gradient of $V$ with respect to $X$. Here $F(X)\cdot DV(X)$ is the Lie derivative of $V$ with respect to the flow since the usual chain rule gives $\ddt{}V(X(t)) = V(X(t)) \cdot DV(X(t))$. The existence of $V$ satisfying \cref{eq:lyapunovfunction} would imply that the subset of $\Omega$ on which $V(X)=0$ is Lyapunov stable, meaning that trajectories remain arbitrarily close to this set if they start sufficiently close \citep{khalil2002nonlinear}. Until about 25 years ago, there was no broadly applicable method for finding $V$ subject to pointwise inequalities like those in~\cref{eq:lyapunovfunction}. Now, however, it is often possible to construct such $V$ computationally using optimization methods based on sum-of-squares (SOS) polynomials \citep{parrilo2000structured, papachristodoulou2002construction}. In particular, if the ODE right-hand side $F$ is polynomial in the components of $X$, and if $V$ is sought from a finite-dimensional space of polynomials, then $F\cdot DV$ also belongs to a finite-dimensional polynomial space. Deciding whether the polynomials $V$ and $-F\cdot DV$ are nonnegative on $\Omega$ is prohibitively hard in general \cite{murty1985some}, but this nonnegativity can be ensured by stronger and more tractable SOS conditions. For instance, simply requiring $V$ to be SOS, meaning that it can be represented as a sum of squares of other polynomials, would imply $V\ge0$ on all of $\mathbb{R}^n$.  Other SOS conditions can imply nonnegativity on $\Omega$ rather than on $\mathbb{R}^n$, as explained in \cref{sec:sos}.

The sphere projection approach that we present in \cref{sec:sphere} uses conditions on an auxiliary function $V$ for bounding time averages. For any $\Phi\in C^0(\Omega)$ evaluated along a trajectory $X(t)$ with initial condition $X_0$, define the infinite-time average $\overline\Phi$ as
\begin{equation}
\overline{\Phi} (X_0) = \limsup_{T\to\infty}\frac{1}{T}\int_0^T\Phi(X(t))\,\mathrm{d}t.
\label{eq:overline}
\end{equation}
On each $X(t)$ that remains bounded forward in time, boundedness of $V(X(t))$ implies $\overline{\ddt{}V}=0$, and thus $\overline{F \cdot DV}=0$. Therefore, on each bounded trajectory in $\Omega$,
\begin{equation}
    \overline{\Phi} = \overline{\Phi+F \cdot DV}\le \sup_{X\in\Omega}[\Phi+F\cdot DV].
    \label{eq:Phi-ineq}
\end{equation}
This upper bound is useful because it does not involve individual trajectories. Assuming each trajectory in $\Omega$ is bounded for $t\ge0$,  one can maximize the left-hand side over $X_0$ and minimize the right-hand side over $V$ to find
\begin{equation}
\label{eq:Phi-opt}
    \sup_{X_0\in \Omega}\overline{\Phi} \le \inf_{V\in C^1(\Omega)} \sup_{X\in \Omega}[ \Phi + F\cdot DV ].
\end{equation}
Furthermore, the inequality in \cref {eq:Phi-opt} is an \emph{equality} when the forward-invariant set $\Omega$ is compact \cite{tobasco2018optimal,Bochi2018}, and this equality underlies sharpness of our sphere projection approach. When $\Phi$, $V$ and $F$ are polynomial, so is the expression $\Phi+F \cdot DV$ on the right-hand side of  \cref{eq:Phi-opt}. In this case, the min--max problem can be relaxed into a minimization subject to SOS constraints, as explained in \cref{sec:sos}. The relaxation is an SOS optimization problem whose objective is to minimize the resulting upper bound on $\overline\Phi$. These ideas were used by \citet{oeri2023convex} to bound the single largest LE. Here we extend this approach to sums of LEs and to Lyapunov dimension.

Our two approaches are the first broadly applicable computational methods that can give convergent upper bounds on global maxima of LE sums or Lyapunov dimension. Unlike methods that rely on pointwise optimization over state space \citep{doering1995shape}, our methods remain sharp when the maximizing orbits are periodic, rather than stationary. Furthermore, our methods give a systematic way to search for the maximizing trajectories themselves, as demonstrated in \cref{sec:duffing}.

The rest of the paper is organized as follows. \Cref{sec:continuous} defines sums of LEs and Lyapunov dimension and then describes our two general approaches to bounding these quantities. After arriving at more general optimization problems, we strengthen their constraints to obtain SOS optimization problems, and then we describe how to exploit symmetries. \Cref{sec:examples} presents computational results from applications of our methods to two examples: the chaotic Duffing oscillator, which can be formulated as an autonomous system in 4 variables, and a 3-degree-of-freedom Hamiltonian system adapted from~\cite{palacian2017periodic}, which is an autonomous system in 6 variables. For all LE sums in both examples, we confirm that our computed bounds are sharp to several digits by finding periodic orbits on which the bounds are approximately saturated. Concluding remarks are given in \cref{sec:conclusion}. For completeness, \cref{sec:compounds} summarizes necessary material about exterior products, and \cref{sec:appendix2} proves a non-standard variant of an existence theorem for Lyapunov functions.

\section{Optimization problems and SOS relaxations}
\label{sec:continuous}

This section defines the quantities we seek to bound and formulates our methods for doing so. \Cref{sec:LE defs} gives expressions for LE sums and Lyapunov dimension that are best suited to what follows. In \cref{sec:sphere} we describe approaches for bounding these quantities based on projecting tangent space vectors onto the unit sphere. In \cref{sec:shifted} we describe approaches that instead shift the spectrum of the Jacobian so that tangent vectors remain bounded. \Cref{sec:sos} describes how each formulation is relaxed into SOS optimization problems that can be implemented numerically. Finally, \cref{sec:syms} describes how symmetries may be exploited in the optimization problems.

\subsection{Tangent vectors, Lyapunov exponents and Lyapunov dimension}
\label{sec:LE defs}

For the remainder of this paper, the ODE system whose LEs are to be studied is denoted by
\begin{equation}
\ddt{}x(t) = f(x(t)), \qquad x(0)=x_0,
\label{eq:dynsys}
\end{equation}
which is the same as~\cref{eq:ODE} but with lowercase $x$ and $f$. The domain of $x(t)$ is denoted by $\mathcal{B}\subset\mathbb{R}^n$, so $f\in C^1(\mathcal{B},\mathbb{R}^n)$. The notation in~\cref{eq:ODE} with uppercase $X$ and $F$, and domain $\Omega$, will be reserved for ODEs associated with~\cref{eq:dynsys} on augmented state spaces that combine the dynamics of $x$ with dynamics on its tangent space. Trajectories are assumed to remain in $\mathcal{B}$ for all $t\ge0$, meaning $\mathcal{B}$ is forward invariant. In cases where $\mathcal{B}$ is unbounded, we further assume that each trajectory remains bounded for $t\ge0$, meaning there is no forward-time blowup. The differentiability of $f$ and the assumption of no blowup guarantee well-posedness for all positive times, meaning there exists a differentiable flow map $\phi:\mathbb{R}_+\times\mathcal{B}\to\mathcal{B}$ satisfying $x(t) = \phi^t(x_0)$ for every initial condition $x_0\in\mathcal{B}$ and all $t\ge0$.

The LEs associated with a trajectory $x(t)$ quantify the convergence or divergence of infinitesimally nearby trajectories. Linearization of~\cref{eq:dynsys} around $x(t)$ defines the dynamics of the tangent vector $y(t)$, which evolves in the tangent space $\mathbb{R}^n$ according to
\begin{equation}
    \ddt{}y(t) = Df(x(t))\,y(t), \qquad y(0)=y_0,
    \label{eq:y ode}
\end{equation}
where $Df(x)$ denotes the $n\times n$ Jacobian matrix of $f$ at $x$, and $y_0$ is the initial tangent vector. Equivalently, $y(t)$ is given by the linearized finite-time flow map as 
\begin{equation}
    y(t)=D\phi^t(x_0)\,y_0.
    \label{eq:y flow map}
\end{equation}
We can restrict $y_0$ to the unit sphere,
\begin{equation}
\label{eq:S notation}
\mathbb{S}^{n-1} = \{ x \in \mathbb{R}^n \,:\,|x| = 1\},
\end{equation}
because the tangent vector dynamics are linear in $y$, so scaling $y_0$ simply scales $y(t)$ proportionally. Considering the separation between trajectories starting at $x_0$ and at $x_0+\varepsilon y_0$, Taylor expanding the flow map and using~\cref{eq:y flow map} and $|y_0|=1$ gives 
\begin{align}
    |\phi^t(x_0)-\phi^t(x_0+\varepsilon y_0)| = \varepsilon |y(t)| + O(\varepsilon^2).
    \label{eq:nearby}
\end{align}
The tangent vector typically shrinks or grows exponentially like $|y(t)|\approx e^{\mu t}$ for some average rate $
\mu(x_0,y_0)$, in which case \cref{eq:nearby} suggests
\begin{equation}
    |\phi^t(x_0)-\phi^t(x_0+\varepsilon y_0)|\approx \varepsilon e^{\mu(x_0,y_0) t}.
    \label{eq:nearby mu}
\end{equation}
If $\mu<0$ this approximation can be valid for all time. If $\mu>0$ this approximation can be valid only until the separation between trajectories grows too large for  \cref{eq:nearby} to apply, but this time horizon approaches infinity as $\varepsilon\to0$. Motivated by the expectation that $|y(t)|\approx e^{\mu t}$, the corresponding LE can be defined precisely as
\begin{equation}
    \mu(x_0,y_0)=\limsup_{t\to\infty}\frac1t \log|y(t)|.
    \label{eq:LE x0 y0}
\end{equation}
For each trajectory $x(t)$, which is determined by its initial condition $x_0$, the LE $\mu$ defined by \cref{eq:LE x0 y0} generally depends on the initial tangent vector direction $y_0$. The Oseledets ergodic theorem guarantees that different $y_0$ give at most $n$ different LEs \citep{pikovsky2016lyapunov}, the largest of which is called the leading LE. \Cref{def:leadingle} below gives an expression for the leading LE. The literature on LEs uses various definitions that are equivalent in many settings \citep{pikovsky2016lyapunov}. 

\begin{defn}
\label{def:leadingle}
For $x(t)$ evolving by the ODE \cref{eq:dynsys}, the \emph{leading Lyapunov exponent} is
\begin{equation}
\label{eq:leadingle}
\mu_1(x_0) = \sup_{y_0\in\mathbb{S}^{n-1}}\limsup_{t\to\infty} \frac{1}{t}\log{|y(t)|},
\end{equation}
where the tangent vector $y(t)$ evolves by \cref{eq:y ode}.
\end{defn}

The leading LE is the average growth rate of lengths in the tangent space, maximized over the initial direction $y_0$. A natural generalization is to consider growth rates of $k$-dimensional volumes in the tangent space for any $k\le n$. From this one can define leading sums of LEs and, in turn, the full spectrum of LEs \cite{temam}. Let $|y_1\wedge\cdots\wedge y_k|$ denote the volume of the parallelepiped whose edges coincide with vectors $y_1,\ldots,y_k\in\mathbb R^n$. If each $y_i$ is a time-evolving tangent vector, the volume of the parallelepiped generally grows or shrinks exponentially in time, so we can define an average exponential rate analogously to \cref{eq:LE x0 y0}. Maximizing this rate over all initial directions $y_i(0)$ of the tangent vectors gives \cite{temam}
\begin{equation}
\label{eq:sumdef}
\Mu_k(x_0) = \sup_{|y_1\wedge \dots \wedge y_k|_{t=0}=1}\limsup_{t\to\infty} \frac{1}{t}\log{|y_1(t)\wedge\dots\wedge y_k(t)|},
\end{equation}
where each $y_i(t)$ evolves as in~\cref{eq:y ode}. The value $\Mu_k(x_0)$ is the maximum time-averaged growth rate of $k$-dimensional volumes in the tangent space for the trajectory $x(t)$ starting at $x_0$. This maximum growth rate will be attained by almost every choice of initial tangent vectors, provided that certain multiplicative ergodic theorems apply \cite{pikovsky2016lyapunov}. 

The set of all `$k$-blades' $y_1\wedge\cdots\wedge y_k$ span a vector space called the exterior product space of $k$-vectors on $\mathbb{R}^n$. Properties of exterior products that we require are reviewed in \cref{sec:compounds}. The space of $k$-vectors has dimension $\bn{n}{k}=\tfrac{n!}{k!(n-k)!}$, so we can identify each $k$-blade $y_1\wedge\cdots\wedge y_k$ with a vector denoted by $y^k\in\mathbb{R}^{\bn{n}{k}}$, whose Euclidean norm $|y^k|$ gives the volume of the parallelepiped with edges $y_1,\ldots,y_k$. For time-dependent $k$-blades where each $y_i(t)$ is a tangent vector that satisfies the ODE~\cref{eq:y ode}, the corresponding ODE that $y^k$ satisfies is
\begin{equation}
    \ddt{}y^k(t) = Df^{[k]}(x(t))\,y^k(t), \qquad y^k(0)=y^k_0. 
    \label{eq:yk ode}
\end{equation}
Here $Df^{[k]}$ denotes the $k^\mathrm{th}$ additive compound matrix induced by the matrix $Df$. This compound is a linear transformation acting on the space $\mathbb{R}^{\bn{n}{k}}$ of $k$-vectors, induced by the linear transformation $Df$ on $\mathbb{R}^n$, and it can be computed explicitly as a $\bn{n}{k}\times \bn{n}{k}$ matrix. Additive (and multiplicative) compound matrices are reviewed in \cref{sec:compounds}. The $y^k$ dynamics~\cref{eq:yk ode} follow from the $y$ dynamics~\cref{eq:y ode} after applying the product and chain rule to $y_1\wedge\dots\wedge y_k$ and using the second part of \cref{thm:addcompoundprops}. The expression~\cref{eq:sumdef} for $\Mu_k$ can be equivalently stated using $y^k$, yielding the following \cref{def:Sigmak} that we choose for~$\Mu_k$.

\begin{defn}
\label{def:Sigmak}
For $x(t)$ evolving by the ODE \cref{eq:dynsys}, the Lyapunov exponents are defined such that the \emph{sum of the $k$ leading Lyapunov exponents} is
\begin{equation}
\Mu_k(x_0) = \sup_{y^k_0\in\mathbb{S}^{{\bn{n}{k}}-1}}\limsup_{t\to\infty} \frac{1}{t}\log{|y^k(t)|},
\label{eq:Sigmak}
\end{equation}
where $y^k(t)$ evolves by \cref{eq:yk ode}.
\end{defn}

Having defined $\Mu_k$, we can define the LE spectrum $\mu_1\ge\mu_2\ge\cdots\ge\mu_n$ such the $\Mu_k$ are necessarily LE sums. When $k=1$ we already have $\mu_1=\Mu_1$ since \cref{def:Sigmak} reduces to \cref{def:leadingle} when $y^k=y$ and $Df^{[k]}=Df$. For $k\ge2$, the LE $\mu_k$ can be defined recursively by
\begin{equation}
    \mu_k=\Mu_k-\Mu_{k-1},
\end{equation}
so indeed $\Mu_k=\sum_{i=1}^k \mu_i$. There are equivalent ways to define the $\mu_k$ directly using singular values, and then define the $\Mu_k$ as their sums \citep{pikovsky2016lyapunov}, but for our purposes it is simpler to define $\Mu_k$ directly by~\cref{eq:Sigmak}.

For systems with attractors, a primary use of LE sums is to define and study the dimension of an attractor. A notion of dimension for fractal sets in general is the Hausdorff dimension, which can be very hard to approximate for fractals that arise from dynamical systems. However, Hausdorff dimension is bounded above by Lyapunov dimension~\cite{eden1991local}, which is given by the Kaplan--Yorke formula \citep{kaplan1979chaotic} in \cref{def:ldim}. (See \citep{kuznetsov2020attractor} for equivalent definitions of Lyapunov dimension.) Our methods give upper bounds on Lyapunov dimension of global attractors, which also are upper bounds on Hausdorff dimension of global attractors. 
\begin{defn}
\label{def:ldim}
Let $x(t)$ evolve by the ODE \cref{eq:dynsys}, and let $\mathcal{B}$ be forward invariant. Define $j$ as the smallest integer such that $\sup_{x_0\in\mathcal{B}}\Mu_{j+1}(x_0)<0$, and define the \emph{global Lyapunov dimension} over $\mathcal{B}$ as
\begin{equation}
\label{eq:dL}
d_L = j + \sup_{x_0\in\mathcal{B}} \left(\frac{\Mu_j(x_0)}{-\mu_{j+1}(x_0)}\right).
\end{equation}
\end{defn}

For particular ODEs, the Kaplan--Yorke formula is usually evaluated numerically along a chaotic trajectory, rather than being maximized over all trajectories. This approximates the fractal dimension of the chaotic attractor, rather than the global attractor. Global Lyapunov dimension pertains to the latter, and the so-called Eden conjecture \cite{eden1989abstract, leonov1993eden} predicts that the supremum in~\cref{eq:dL} is attained on an equilibrium or periodic orbit, not a chaotic orbit. \Cref{thm:dimlemma} states that $j\le d_L<j+1$, which we prove for completeness, and gives a rearranged formula for $d_L$ that we use in \cref{sec:sphere,sec:shifted}.

\begin{lem}
\label{thm:dimlemma}
The Lyapunov dimension $d_L$ has the following properties.
\begin{enumerate}
\item For the integer $j$ in \cref{def:ldim}, Lyapunov dimension is bounded by $j\le d_L<j+1$.
\item The expression~\cref{eq:dL} for $d_L$ is equivalent to
\begin{equation}
\label{eq:dL restatement}
d_L = \inf_{B\in[0,1)}(j+B) \quad \text{s.t.}\quad 
B\left[\Mu_j(x_0)-\Mu_{j+1}(x_0)\right]-\Mu_j(x_0)\ge 0~~\forall~x_0\in\mathcal{B}.
\end{equation}
\end{enumerate}
\end{lem}
\begin{proof}
The definition of $j$ implies that there exists $\varepsilon>0$ such that
\begin{equation}
\Mu_{j+1}(x_0)\le -\varepsilon \quad\text{and} \quad -\mu_{j+1}(x_0)\ge \varepsilon\quad\text{for all}~x_0\in\mathcal{B},
\end{equation}
and that $\sup_{x_0\in\mathcal{B}}\Mu_j(x_0)\ge 0$. From these facts it follows that the supremum in the definition~\cref{eq:dL} of $d_L$ cannot be negative, which gives the lower bound $j\le d_L$ in the first part of the lemma. For the upper bound, note that 
\begin{equation}
\label{eq:dL upper bound}
\frac{\Mu_j(x_0)}{-\mu_{j+1}(x_0)}\le \frac{-\mu_{j+1}(x_0)-\varepsilon}{-\mu_{j+1}(x_0)} < 1,
\end{equation}
which gives $d_L<j+1$.

For the second part of the lemma, note that the supremum in the definition~\cref{eq:dL} of $d_L$ is equivalent to
\begin{equation}
\label{eq:dL lemma pf}
\inf B \quad \text{s.t.}\quad \frac{\Mu_j(x_0)}{-\mu_{j+1}(x_0)}\le B~~\forall~x_0\in\mathcal{B}.
\end{equation}
Since $-\mu_{j+1}=\Mu_j-\Mu_{j+1}>0$, rearranging the constraint in~\cref{eq:dL lemma pf} gives the constraint in~\cref{eq:dL restatement} as an equivalent form. This infimum is in $[0,1)$ and so is unchanged by the $B\in[0,1)$ constraint in \cref{eq:dL restatement}. Thus expression~\cref{eq:dL restatement} for $d_L$ is equivalent to~\cref{eq:dL}.
\end{proof}

\begin{rmk}
    \label{rmk:dL upper bound}
    Combining both parts of \cref{thm:dimlemma} implies that, if we can find $B\in\mathbb{R}$ and $k\geq j$ where
    \begin{equation}
    \label{eq:dL rmk inequality}
B\left[\Mu_k(x_0)-\Mu_{k+1}(x_0)\right]-\Mu_k(x_0)\ge 0
\end{equation}
for all $x_0\in\mathcal{B}$, then $d_L\leq k+B$. Note that~\cref{eq:dL rmk inequality} cannot hold for all $x_0$ unless $B\ge0$.
\end{rmk}

\subsection{Sphere projection approach}
\label{sec:sphere}

Our first approach to bounding the LE sums $\Mu_k$ uses the formulation~\cref{eq:Phi-opt} for bounding infinite-time averages. This is a generalization of the approach in \cite{oeri2023convex}, which was described only for the leading LE---i.e., the $k=1$ case. First note that the time average in the \cref{def:Sigmak} for $\Mu_k$ can be expressed as a time integral along trajectories in the augmented state space for $(x,y^k)$:
\begin{align}
\limsup_{t\to\infty}\frac{1}{t} \log{|y^k(t)|} 
&= \limsup_{T\to\infty}\frac{1}{T}\int_0^T\ddt{}{\log{|y^k(t)|}}\,\mathrm{d}t \\ 
&= \limsup_{T\to\infty}\frac{1}{T}\int_0^T \frac{(y^k)^\mathsf{T}\ddt{}y^k}{|y^k|^2}\,\mathrm{d}t\\
&= \limsup_{T\to\infty}\frac{1}{T}\int_0^T
 \frac{(y^k)^\mathsf{T}Df^{[k]}(x)y^k}{|y^k|^2}\,\mathrm{d}t, \label{eq:Phik y}
\end{align}
where $(y^k)^\mathsf{T}$ is the transpose of the vector $y^k$. The formulation~\cref{eq:Phi-opt} for bounding time averages could be applied to the quantity in \cref{eq:Phik y} if each $(x,y^k)$ were to remain bounded for $t\ge0$, but $y^k(t)$ will grow without bound when $\Mu_k$ is positive. Instead, the tangent space dynamics can be captured with bounded variables by projecting $y^k$ onto a sphere, by introducing $z^k=y^k/|y^k|$. Expressing the right-hand side of  \cref{eq:Phik y} in terms of $z^k$ and then maximizing both sides over the initial tangent space directions $z^k_0$ gives
\begin{equation}
   \Mu_k(x_0) = \sup_{z_0^k\in\mathbb{S}^{\bn{n}{k}-1}}\overline{\Phi_k},
   \label{eq:Sigmak Phik}
\end{equation}
where the overline denotes a time average of
\begin{equation}
 \Phi_k(x,z^k) = (z^k)^\mathsf{T} Df^{[k]}(x)z^k
 \label{eq:Phik}
\end{equation}
along trajectories in the $(x,z^k)$ state space. The ODE for $z^k$ follows from the ODE~\cref{eq:yk ode} for $y^k$, and it couples with the ODE~\cref{eq:dynsys} for $x$ to form the augmented system
\begin{equation}
\label{eq:projectedaugmented}
\ddt{} \begin{bmatrix}x\\z^k\end{bmatrix} =\begin{bmatrix}f(x)\\\ell_k(x,z^k)\end{bmatrix}
\end{equation}
on state space $\mathcal{B}\times\mathbb{S}^{\bn{n}{k}-1}$, where
\begin{equation}
\label{eq:ell}
\ell_k(x,z^k)= Df^{[k]}(x) z^k - \Phi_k(x,z^k)z^k.
\end{equation}
Expressions for $\Mu_k$ in terms of the projected tangent space dynamics~\cref{eq:projectedaugmented} have often appeared in the study of stochastic dynamics, where they are called Furstenberg–Khasminskii formulas \cite[chapter 6]{Arnold_1998}. In the stochastic case, the formula~\cref{eq:Sigmak Phik} differs by additional terms in $\Phi_k$, and the averaging is with respect to stationary measures.

Auxiliary function methods for bounding time averages can be applied to $\Mu_k$ by using its representation~\cref{eq:Sigmak Phik} as a time average in the augmented state space for $(x,z^k)$. In particular, we directly apply the formula~\cref{eq:Phi-opt} with state variable $X=(x,z^k)$ on domain $\mathcal{B}\times\mathbb{S}^{\bn{n}{k}-1}$ and with $F(X)$ being the right-hand side of~\cref{eq:projectedaugmented}. The resulting inequality is the first part of \cref{thm:projectionmethod} below, for which no further proof is needed. If $F\in C^1$ and the domain is compact, then~\cref{eq:Phi-opt} holds with equality, as proved in \cite{tobasco2018optimal}. This yields the second part of \cref{thm:projectionmethod} since assuming  $f\in C^2$ ensures $F\in C^1$. The $k=1$ special case of \cref{thm:projectionmethod} is proposition 1 in~\cite{oeri2023convex}.
\begin{prop}
\label{thm:projectionmethod}
Let $x(t)$ evolve by the ODE \cref{eq:dynsys}, where $f\in C^1(\mathcal{B},\mathbb{R}^n)$ and $\mathcal{B}$ is forward invariant. Let $ \Phi_k(x,z^k)$ and $\ell_k(x,z^k)$ be as defined by \cref{eq:Phik} and \cref{eq:ell}, respectively.
\begin{enumerate}
\item Suppose each trajectory $x(t)$ in $\mathcal{B}$ is bounded for $t\ge0$. For each $k\in\{1,\ldots,n\}$, the maximum leading LE sum $\Mu_k$ among trajectories in $\mathcal{B}$ is bounded above by
\begin{equation}
\label{eq:Sigmak C1 ineq}
\sup_{x_0\in \mathcal{B}} \Mu_k(x_0) \leq \inf_{V\in C^1\left( \mathcal{B}\times\mathbb{R}^{\bn{n}{k}}\right)}\sup_{\substack{x\in  \mathcal{B}\hspace{13pt}\\z^k\in\mathbb{S}^{{\bn{n}{k}}-1}}}\left[\Phi_k + f\cdot D_xV+\ell_k\cdot D_{z^k}V\right],
\end{equation}
where $D_xV$ and $D_{z^k}V$ denote the gradients of $V(x,z^k)$ with respect to $x$ and $z^k$, respectively.
    \item Suppose $\mathcal{B}$ is compact and $f\in C^2(\mathcal{B},\mathbb{R}^n)$. Then~\cref{eq:Sigmak C1 ineq} holds with equality.
\end{enumerate}
\end{prop}
The optimization problem on the right-hand side of~\cref{eq:Sigmak C1 ineq} may be intractable, but it can be relaxed into computationally tractable optimization problems that also yield upper bounds on $\Mu_k$. Relaxations using SOS polynomials are given in \cref{sec:sos} below. 
 
The Lyapunov dimension $d_L$ can be bounded above using similar ideas. One way is to apply upper bounds $\Mu_k\le B_k$ that have already been found using the right-hand side of \cref{eq:Sigmak C1 ineq} or its SOS relaxations. If such bounds are computed with $B_k\ge0$ but $B_{k+1}<0$ for some integer $k$, then \cref{thm:dimension_separate} below gives an upper bound on $d_L$. To apply this result, one does not need to know whether $k$ is the minimal integer $j$ appearing in \cref{def:ldim} for $d_L$. It is sufficient that $B_{k+1}<0$, which implies $k\ge j$.

\begin{prop}
    \label{thm:dimension_separate}
Suppose that, for some $k\in\{1,\ldots,n-1\}$ and values $B_k\ge0$ and $B_{k+1}<0$, 
\begin{equation}
\label{eq:Bk assumed}
 \sup_{x_0\in\mathcal{B}}\Mu_k(x_0)\le B_k \quad \text{and}\quad\sup_{x_0\in\mathcal{B}} \Mu_{k+1}(x_0)\le B_{k+1}.
\end{equation}
Then, the global Lyapunov dimension over $\mathcal{B}$ (\cref{def:ldim}) is bounded above by
    \begin{equation}
    \label{eq:dimbound1}
        d_L\leq k+\frac{B_k}{B_k-B_{k+1}}.
    \end{equation}
\end{prop}
\begin{proof}
The meaning of $j$ in \cref{def:ldim} and the fact that $B_{k+1}<0$ together imply $k\ge j$. Due to \cref{rmk:dL upper bound}, it suffices for the inequality \cref{eq:dL rmk inequality} to hold for all $x_0$ with $B=\frac{B_k}{B_k-B_{k+1}}$.
    For any $x_0\in \mathcal{B}$, the assumptions give $\Mu_k(x_0)\le B_k$ and $0<-B_{k+1}\le-\Mu_{k+1}(x_0)$, so
    \begin{equation}
    \label{eq:proof inequality}
  -B_{k+1}\Mu_k(x_0)\le   -B_k \Mu_{k+1}(x_0),
    \end{equation}
    and
\begin{equation}
B_k\Mu_k(x_0) -B_{k+1}\Mu_k(x_0) \le  B_k\Mu_k(x_0)-B_k \Mu_{k+1}(x_0).
\end{equation}
Dividing by the positive quantity $B_k-B_{k+1}$ and rearranging gives
    \begin{equation}
    \label{eq:prop_intermediate}
            \frac{B_k}{B_k-B_{k+1}}\left[\Mu_k(x_0)-\Mu_{k+1}(x_0)\right]-\Mu_k(x_0)\geq 0,
    \end{equation}
  which is the required inequality~\cref{eq:dL rmk inequality}.
\end{proof}

The upper bound on Lyapunov dimension given by \cref{thm:dimension_separate} cannot be sharp for certain systems, even if $B_k$ and $B_{k+1}$ are each sharp bounds in~\cref{eq:Bk assumed}. The reason is that the inequality \cref{eq:proof inequality}, when both sides are maximized over $x_0$, is an equality only if $\Mu_k(x_0)$ and $\Mu_{k+1}(x_0)$ attain their maxima on the same orbits. This is not true in general; see \cref{sec:hamiltonian} for a counterexample. We thus give a second formulation for upper bounds on $d_L$ that is generally sharper than \cref{thm:dimension_separate}. This formulation directly considers the ratio in the Kaplan--Yorke formula~\cref{eq:dL}, rather than separately bounding the $\Mu_k$ and $\Mu_{k+1}$ terms in that ratio. In particular, we consider the $x$ dynamics together with the projected dynamics of $k$-vectors and $(k+1)$-vectors,
\begin{equation}
\label{eq:augmenteddim1}
\ddt{} \begin{bmatrix}x\\z^k\\z^{k+1}\end{bmatrix} = \begin{bmatrix}f(x)\\\ell_k(x,z^k)\\\ell_{k+1}(x,z^{k+1})\end{bmatrix},
\end{equation}
where $\ell_k$ and $\ell_{k+1}$ are defined by \cref{eq:ell}. The augmented state space of this system is $(x,z^k,z^{k+1})\in\mathcal{B}\times\mathbb{S}^{\bn{n}{k}-1}\times\mathbb{S}^{\bn{n}{k+1}-1}$. The constraint in \cref{eq:dL restatement} from \cref{thm:dimlemma} can then be expressed in terms of trajectories of the ODE system \cref{eq:augmenteddim1}. For this we define
\begin{equation}
\label{eq:PsiB}
\Psi^B(x,z^k,z^{k+1}) = (1-B)\Phi_k(x,z^k)+B\,\Phi_{k+1}(x,z^{k+1}),
\end{equation}
and we consider the condition
\begin{equation}
\overline{\Psi}^B(x_0,z^k_0,z^{k+1}_0)\leq 0 
~~\text{for all}~
(x_0,z^k_0,z^{k+1}_0)\in\mathcal{B}\times\mathbb{S}^{\bn{n}{k}-1}\times\mathbb{S}^{\bn{n}{k+1}-1},
\label{eq:PsiB cond}
\end{equation}
where the time average denoted by the overbar is along trajectories of \cref{eq:augmenteddim1}. Condition~\cref{eq:PsiB cond} is equivalent to the condition in \cref{eq:dL restatement} of \cref{thm:dimlemma}, as explained below in the proof of \cref{thm:dimcombined}. The condition~\cref {eq:PsiB cond} can be enforced using the general formulation~\cref{eq:Phi-opt} for bounding time averages, applied to the augmented system \cref{eq:augmenteddim1}. This yields the first part of \cref{thm:dimcombined}, whose sharpness is addressed by its second part. \Cref{thm:dimcombined} can be applied only when the LE sum $\Mu_{k+1}$ for certain $k$ is known to be negative on every trajectory. Such negativity can be confirmed by using \cref{eq:Sigmak C1 ineq}, or its SOS relaxation that we introduce in \cref{sec:sos}, to obtain a negative upper bound on $\Mu_{k+1}$.
\begin{prop}
\label{thm:dimcombined}
Let $x(t)$ evolve by the ODE \cref{eq:dynsys}, where $f\in C^1(\mathcal{B},\mathbb{R}^n)$ and $\mathcal{B}$ is forward invariant. Let $ \Phi_k(x,z^k)$ and $\ell_k(x,z^k)$ be as defined by~\cref{eq:Phik} and \cref{eq:ell}, and likewise for  $ \Phi_{k+1}(x,z^{k+1})$ and $\ell_{k+1}(x,z^{k+1})$. 
\begin{enumerate}
\item  Suppose each trajectory $x(t)$ in $\mathcal{B}$ is bounded for $t\ge0$. Let the positive integer $k< n$ be such that $\sup_{x_0\in\mathcal{B}}\Mu_{k+1}(x_0)<0$. Then the global Lyapunov dimension $d_L$ over $\mathcal{B}$ is bounded above by
\begin{multline}
\label{eq:dL bounds 1}
d_L\le \inf\limits_{\substack{B\in[0,1)\\V\in C^1~~}}(k+B) 
 \quad\text{s.t.}\quad \begin{array}[t]{l}
\Psi^B + f\cdot D_xV+\ell_k\cdot D_{z^k}V +\ell_{k+1}\cdot D_{z^{k+1}}V\le0 \\[4pt]
\text{for all}~~ 
(x,z^k,z^{k+1})\in\mathcal{B}\times\mathbb{S}^{\bn{n}{k}-1}\times\mathbb{S}^{\bn{n}{k+1}-1},
\end{array}
\end{multline}
where $V:\mathcal{B}\times\mathbb{S}^{{\bn{n}{k}}-1}\times\mathbb{S}^{\bn{n}{k+1}-1}\to\mathbb{R}$, and $D_xV$, $D_{z^k}V$ and $D_{z^{k+1}}V$ denote the gradients of $V(x,z^k,z^{k+1})$ with respect to each argument.
\item Suppose $\mathcal{B}$ is compact and $f\in C^2(\mathcal{B},\mathbb{R}^n)$. Let $j$ be the minimum admissible $k$, as in \cref{def:ldim} for $d_L$. Then, for all $\varepsilon>0$,
\begin{equation}
\label{eq:dL bounds 2}
d_L~-~\varepsilon\,\left|\sup_{x_0\in\mathcal{B}}\Mu_{j+1}(x_0)\right|^{-1}
\le j+B_\varepsilon\le  d_L,
\end{equation}
where
\begin{equation}
\label{eq:Beps}
B_\varepsilon = \inf\limits_{\substack{B\in[0,1)\\V\in C^1\hspace{7pt}}}B
\quad \text{s.t.}\quad \begin{array}[t]{l}
\Psi^B + f\cdot D_xV+\ell_j\cdot D_{z^j}V +\ell_{j+1}\cdot D_{z^{j+1}}V\le \varepsilon \\[4pt]
\text{for all}~~ 
(x,z^k,z^{k+1})\in\mathcal{B}\times\mathbb{S}^{\bn{n}{k}-1}\times\mathbb{S}^{\bn{n}{k+1}-1}.
\end{array}
\end{equation}
\end{enumerate}
\end{prop}
\begin{proof}
For the first part, the assumption that $\sup_{x_0\in\mathcal{B}}\Mu_{k+1}(x_0)<0$ means that $k\geq j$, where $j$ is as in \cref{def:ldim}. By \cref{rmk:dL upper bound}, we therefore have
\begin{equation}
\label{eq:dL bound pf}
    d_L \le \inf_{B\in[0,1)}(k+B) \quad \text{s.t.}\quad B\left[\Mu_k(x_0)-\Mu_{k+1}(x_0)\right]-\Mu_k(x_0)\geq 0~~\text{for all}~x_0\in\mathcal{B}.
\end{equation}
The constraint in \cref{eq:dL bound pf} is equivalent to the $\overline{\Psi}^B\le 0$ condition~\cref{eq:PsiB cond}. To see this, we rearrange the constraint in~\cref{eq:dL bound pf} and use the expression~\cref{eq:Sigmak Phik} for $\Mu_k$ to obtain an equivalent inequality,
\begin{equation}
(1-B)\left[\sup_{z_0^k\in\mathbb{S}^{\bn{n}{k}}}\overline\Phi_k(x_0,z_0^k)\right]+ B\left[\sup_{z_0^{k+1}\in\mathbb{S}^{\bn{n}{k+1}}}\overline\Phi_{k+1}(x_0,z_0^{k+1})\right]~\le~0
~~\text{for all}~x_0\in\mathcal{B}.
\end{equation}
This condition is equivalent to the $\overline{\Psi}^B\le 0$ condition~\cref{eq:PsiB cond}, so the latter can replace the constraint in~\cref{eq:dL bound pf} to give
\begin{equation}
\label{eq:dL bound pf 2}
d_L\le\inf_{B\in[0,1)}(k+B) \quad \text{s.t.}\quad \overline{\Psi}^B(x_0,z_0^k,z_0^{k+1})\le 0~~ \text{for all}~(x_0,z_0^k,z_0^{k+1})\in\mathcal{B}\times\mathbb{S}^{\bn{n}{k}-1}\times\mathbb{S}^{\bn{n}{k+1}-1}.
\end{equation}

We want to replace the constraint in \cref{eq:dL bound pf 2} with one that does not involve time averages. To do so we bound $\overline{\Psi}^B$ above using the general formulation~\cref{eq:Phi-opt} for time averages by letting $X=(x,z^k,z^{k+1})$ on domain $\Omega=\mathcal{B}\times\mathbb{S}^{\bn{n}{k}-1}\times\mathbb{S}^{\bn{n}{k+1}-1}$, and letting $F(X)$ be the right-hand side of~\cref{eq:augmenteddim1}. This gives
\begin{equation}
\label{eq:dL sphere pf 1}
\sup_{\substack{x_0\in\mathcal{B}\hspace{18pt}\\z^k_0\in\mathbb{S}^{\bn{n}{k}-1}\\z^{k+1}_0\in\mathbb{S}^{\bn{n}{k+1}-1}}}
\overline{\Psi}_B
\le \inf_{V\in C^1}\sup_{\substack{x\in\mathcal{B}\hspace{14pt}\\z^k\in\mathbb{S}^{\bn{n}{k}-1}\\z^{k+1}\in\mathbb{S}^{\bn{n}{k+1}-1}}}
\left[\Psi^B + f\cdot D_xV+\ell_k\cdot D_{z^k}V +\ell_{k+1}\cdot D_{z^{k+1}}V\right].
\end{equation}
The constraint in \cref{eq:dL bound pf}, which is equivalent to nonpositivity of the left-hand side of~\cref{eq:dL sphere pf 1}, can be strengthened by requiring there to exist $V\in C^1$ for which the inner supremum on the right-hand side is nonpositive. This gives the first part of the proposition.

For the second part, note that~\cref{eq:dL bound pf} is an equality when $k=j$, according to~\cref{thm:dimlemma}, and likewise~\cref{eq:dL bound pf 2} is an equality:
\begin{equation}
\label{eq:dL sphere pf 2}
d_L=\inf_{B\in[0,1)}(j+B) \quad \text{s.t.}\quad \overline{\Psi}^B(x_0,z_0^j,z_0^{j+1})\le 0~~\text{for all}~(x_0,z_0^j,z_0^{j+1})\in\mathcal{B}\times\mathbb{S}^{\bn{n}{j}-1}\times\mathbb{S}^{\bn{n}{j+1}-1}.
\end{equation}
To prove the second inequality in the claim~\cref{eq:dL bounds 2}, let $B$ be any value for which the constraint in~\cref{eq:dL sphere pf 2} holds. It suffices to prove existence of $V$ for which constraint in the definition~\cref{eq:Beps} of $B_\varepsilon$ holds also. This will imply that the infimum over $B$ in~\cref{eq:Beps}  is at least as small as in~\cref{eq:dL sphere pf 2}. 

Under the assumptions that $\mathcal{B}$ is compact and $f\in C^2$, the $X$ domain $\Omega$ is compact, and $F(X)$, the right-hand side of~\cref{eq:augmenteddim1}, is $C^1$ on an open region containing $\Omega$. For such $X$ and $F$, the general upper bound~\cref{eq:Phi-opt} is an equality \cite{tobasco2018optimal}, so in particular~\cref{eq:dL sphere pf 1} is an equality. Equality in~\cref{eq:dL sphere pf 1} with $k=j$, and in~\cref{eq:dL sphere pf 2},  implies that the constraint in~\cref{eq:dL sphere pf 2} is equivalent to
\begin{equation}
\label{eq:dL sphere pf 3}
0\ge \inf_{V\in C^1}\sup_{\substack{x\in\mathcal{B}\\z^k\in\mathbb{S}^{\bn{n}{k}-1}\\z^{k+1}\in\mathbb{S}^{\bn{n}{k+1}-1}}}
\left[\Psi^B + f\cdot D_xV+\ell_k\cdot D_{z^k}V +\ell_{k+1}\cdot D_{z^{k+1}}V\right].
\end{equation}
This inequality implies that, for any $\varepsilon>0$, there exists $V\in C^1$ such that
\begin{equation}
\label{eq:dL sphere pf 4}
\left[\Psi^B + f\cdot D_xV+\ell_k\cdot D_{z^k}V +\ell_{k+1}\cdot D_{z^{k+1}}V\right] \le \varepsilon\quad\forall~(x,z^k,z^{k+1})\in\mathcal{B}\times\mathbb{S}^{\bn{n}{k}-1}\times\mathbb{S}^{\bn{n}{k+1}-1}.
\end{equation}
(Such $V$ need not exist with $\varepsilon=0$ if the infimum over $V$ in~\cref{eq:dL sphere pf 1} is not attained.) Condition~\cref{eq:dL sphere pf 4} is the same as the constraint in the minimization~\cref{eq:Beps} defining $B_\varepsilon$, and this implies $B_\varepsilon\le B$. We have shown that $B_\varepsilon\le B$ for any $B$ satisfying the constraint in~\cref{eq:dL sphere pf 2}, and thus $j+B_\varepsilon\le d_L$. This establishes the second inequality in~\cref{eq:dL bounds 2}.

To prove the first inequality in the claim~\cref{eq:dL bounds 2}, suppose that the inequality constraint in the definition~\cref{eq:Beps} of $B_\varepsilon$ is satisfied for a pair of $B$ and $V$. Averaging this inequality along any trajectory and then maximizing over the initial directions $z_0^k$ and $z_0^{k+1}$ gives
\begin{equation}
\label{eq:dL sphere pf 6}
(1-B)\Mu_k(x_0)+B\Mu_{k+1}(x_0)\le \varepsilon.
\end{equation}
Dividing by the positive quantity $-\mu_{k+1}=\Mu_k-\Mu_{k+1}$ gives
\begin{equation}
\label{eq:dL sphere pf 7}
\frac{\Mu_k(x_0)}{-\mu_{k+1}(x_0)}-\frac{\varepsilon}{-\mu_{k+1}(x_0)}\le B.
\end{equation}
Maximizing the first term over $x_0$, and bounding the second term above using $\mu_{k+1}(x_0)\le \sup_{x_0\in\mathcal{B}}\Mu_{k+1}<0$, we find
\begin{equation}
\label{eq:dL sphere pf 8}
(d_L-j) - \frac{\varepsilon}{\left|\sup_{x_0\in\mathcal{B}}\Mu_{j+1}(x_0)\right|}\le B.
\end{equation}
This inequality holds at every $B$ for which there exists a $V$ satisfying the constraint in the minimization~\cref{eq:Beps} that defines $B_\varepsilon$. Therefore, minimizing over such $B$ gives~\cref{eq:dL sphere pf 8} with $B_\varepsilon$ on the right-hand side. This establishes the first inequality in~\cref{eq:dL bounds 2}, which completes the proof.
\end{proof}
\begin{rmk}
The second part of \cref{thm:dimcombined} implies that
\begin{equation}
d_L\le j + B_\varepsilon + \varepsilon\,\left|\sup_{x_0\in\mathcal{B}}\Mu_{j+1}(x_0)\right|^{-1}
\end{equation}
at each $\varepsilon>0$, and that this upper bound converges to $d_L$ as $\varepsilon\to0$.

\end{rmk}

\subsection{Shifted spectrum approach}
\label{sec:shifted}
In the preceding subsection, the possibly unbounded $k$-vectors $y^k$ are replaced by their projections $z^k$ onto the unit sphere, which are necessarily bounded. In the alternative approach of the present subsection, we replace the dynamics of $y^k$ by
\begin{equation}
\label{eq:wk}
w^k(t) = c\hspace{1pt}e^{-Bt}y^k(t)
\end{equation}
for any nonzero constant $c$, so the initial condition $w_0^k=c\hspace{1pt}y_0^k$ can be any nonzero point in $\mathbb{R}^{\bn{n}{k}}$. This $w^k(t)$ is bounded forward in time for sufficiently large $B$. The ODE for $w^k$ is the same as the ODE~\cref{eq:yk ode} for $y^k$, except the spectrum of the Jacobian is shifted by $-B$. Together with the ODE~\cref{eq:dynsys} for $x$ this gives the augmented system
\begin{equation}
    \label{eq:shiftedaugmented}
    \ddt{}\begin{bmatrix}
        x\\w^k
    \end{bmatrix} = \begin{bmatrix}
        f(x)\\ m^B_k(x,w^k)
    \end{bmatrix},
\end{equation}
where
\begin{equation}
\label{eq:ellB}
m^B_k(x,w^k) = Df^{[k]}(x) w^k - B w^k.
\end{equation}
Note that this ODE is linear in $w^k$, whereas the ODE~\cref{eq:projectedaugmented} in the sphere projection approach is nonlinear in $z^k$. The following lemma asserts that if $B$ is large enough to prevent unbounded growth of $w^k$, then $B$ is an upper bound on the LE sum $\Mu_k$, and the infimum over such $B$ is equal to~$\Mu_k$. The second part gives stronger conditions for this to be true in a uniform sense for all initial conditions.

\begin{lem}
\label{thm:shiftedspectrumstability}
Let $x(t)$ evolve by the ODE \cref{eq:dynsys}, where $f\in C^1(\mathcal{B},\mathbb{R}^n)$ and $\mathcal{B}$ is forward invariant, and let $w^k(t)$ evolve by the ODE~\cref{eq:shiftedaugmented} with $B$ fixed. Assume each trajectory $x(t)$ in $\mathcal{B}$ is bounded for $t\ge0$. For each $k\in\{1,\ldots,n\}$:
\begin{enumerate}
\item The leading LE sum $\Mu_k(x_0)$ is equivalent to
\begin{equation}
\label{eq:LE sum is inf}
\Mu_k(x_0) = \inf_{B\in\mathbb{R}} B\quad\text{s.t.}\quad \text{for each } w^k_0\in\mathbb{R}^{\bn{n}{k}}, ~w^k(t)\text{ is bounded for }t\ge0.
\end{equation}
\item Suppose $f\in C^2(\mathcal{B},\mathbb{R}^n)$. For each compact set $\mathcal{A}\subset\mathcal{B}$ and each $B>\sup_{x_0\in\mathcal{A}} \Mu_k(x_0)$, there exist constants $K>0$ and $\lambda>0$ such that
\begin{equation}
\label{eq:exponential decay of w}
|w^k(t)|\leq K e^{-\lambda t}|w^k_0| \quad \text{for all $t\geq 0$ and }(x_0,w^k_0)\in\mathcal{A}\times\mathbb{R}^{\bn{n}{k}}.
\end{equation}
\end{enumerate}
\end{lem}
\begin{proof}
We first show that if $B$ satisfies the constraint in \cref{eq:LE sum is inf}, then $\Mu_k(x_0)\leq B$. Substituting $y^k=e^{Bt}w^k/c$ into the time average on the right-hand side of~\cref{eq:Sigmak} gives
\begin{equation}
\limsup_{t\to\infty} \frac{1}{t} \log \left| y^k(t) \right|
=\limsup_{t\to\infty} \frac{1}{t} \log \left( \left| e^{Bt}w^k(t) \right|/c\right) 
    = B+\limsup_{t\to\infty} \frac{1}{t} \log \left|w^k(t) \right| 
     \le B,
\end{equation}
where the final inequality follows from boundedness of $w^k(t)$.
Taking the supremum of the left-hand side over $y^k_0$ gives $\Mu_k(x_0)\le B$. Next, we assume that $B>\Mu_k(x_0)$ and show that the constraint in \cref{eq:LE sum is inf} must hold. With such $B$, and with $\mathcal{A}=\{x_0\}$, the second part of the proposition (proved below) applies because $B>\Mu_k(x_0)=\sup_{x_0'\in\mathcal{A}} \Mu_k(x_0')$. Thus \cref{eq:exponential decay of w} holds and implies that $w^k(t)$ is bounded for all $t\ge 0$, and the first part is proved.

To prove the second part of the proposition, let $B>\sup_{x_0\in\mathcal{A}} \Mu_k(x_0)$. It suffices to show that $K$ and $\lambda$ exists so that \cref{eq:exponential decay of w} holds for $(x_0,w_0^k)\in\mathcal{A}\times\mathbb{S}^{\bn{n}{k}-1}$, rather than all $w^k_0\in\mathbb{R}^{\bn{n}{k}-1}$. These same $K$ and $\lambda$  remain valid for all $w^k_0\in\mathbb{R}^{\bn{n}{k}}$ since $w^k(t)$ scales proportionally to $w^k_0$, due to the linearity of the $w^k$ dynamics in~\cref{eq:shiftedaugmented}.

To choose $\lambda$, note that
\begin{equation}
\begin{aligned}
        \sup_{x_0\in\mathcal{A}}\Mu_k(x_0) &=\sup_{(x_0,y^k_0)\in\mathcal{A}\times\mathbb{S}^{\bn{n}{k}-1}}\limsup_{t\to\infty} \frac{1}{t} \log{|y^k(t)|} \\
        &= B + \sup_{(x_0,w^k_0)\in\mathcal{A}\times\mathbb{R}^{\bn{n}{k}}}\limsup_{t\to\infty} \frac{1}{t} \log{|w^k(t)|},
\end{aligned}
\end{equation}
and so
\begin{equation}
    \sup_{(x_0,w^k_0)\in\mathcal{A}\times\mathbb{R}^{\bn{n}{k}}}\limsup_{t\to\infty} \frac{1}{t} \log{|w^k(t)|} < 0.
\end{equation}
We therefore can choose a $\lambda>0$ such that
\begin{equation}
\limsup_{t\to\infty} \frac{1}{t} \log{|w^k(t)|} < - \lambda \quad\text{ for all }\quad(x_0,w^k_0)\in\mathcal{A}\times\mathbb{R}^{\bn{n}{k}}.
\label{eq: w decay}
\end{equation}

To choose $K$, note that~\cref{eq: w decay} implies $|w^k(t)|\le e^{-\lambda t}|w^k_0|$ after some time $T(x_0,w^k_0)$. We choose $T$ to depend continuously on the initial conditions $(x_0,w^k_0)$, which is possible since continuity of $w^k(t)$ in both $t$ and $(x_0,w^k_0)$ follows from the fact that the right-hand side of the ODE~\cref{eq:shiftedaugmented} is locally Lipschitz when $f\in C^2$. Having chosen $T(x_0,w^k_0)$ to be continuous, the maximum
\begin{equation}
T^* = \max_{(x_0,w^k_0)\in\mathcal{A}\times\mathbb{S}^{\bn{n}{k}-1}}T(x_0,w^k_0)
\end{equation}
is finite since it is over a compact set. Therefore $|w^k(t)| \leq e^{-\lambda t} |w^k_0|$ for all $t>T^*$ and all $(x_0,w_0)\in\mathcal{A}\times\mathbb{S}^{\bn{n}{k}-1}$. Finally, we choose
\begin{equation}
K=\max_{\substack{(x_0,w^k_0)\in\mathcal{A}\times\mathbb{S}^{\bn{n}{k}-1}\\t\in[0,T^*]}} e^{\lambda t} |w^k(t)|,
\end{equation}
which again is finite since it is the maximum of a continuous function over a compact set. For the chosen $K$ and $\lambda$, the inequality in~\cref{eq:exponential decay of w} indeed holds over $\mathcal{A}\times\mathbb{S}^{\bn{n}{k}-1}$ and thus also over $\mathcal{A}\times\mathbb{R}^{\bn{n}{k}}$.
\end{proof}

Practical upper bounds on the LE sum $\Mu_k$ are found by applying a specialized  function theorem (\cref{thm:lyapunovfunction}) to the ODE system \cref {eq:shiftedaugmented} and combining the result with \cref{thm:shiftedspectrumstability}. This yields the following proposition.

\begin{prop}
\label{thm:shiftedmethod}
Let $x(t)$ evolve by the ODE \cref{eq:dynsys}, where $f\in C^1(\mathcal{B},\mathbb{R}^n)$ and $\mathcal{B}$ is forward invariant. Let $m^B_k(x,w^k)$ be as defined in \cref{eq:ellB}.
\begin{enumerate}
\item  Suppose each trajectory $x(t)$ in $\mathcal B$ is bounded for $t\ge0$. For each $k\in\{1,\ldots,n\}$, the maximum leading LE sum $\Mu_k$ is bounded above by
\begin{equation}
\label{eq:Sigmak C1 ineq shifted}
\sup_{x_0\in \mathcal{B}} \Mu_k(x_0) \leq \inf_{V\in C^1(\mathcal{B}\times\mathbb{R}^{\bn{n}{k}})} B \quad
\text{s.t.}\quad \begin{array}[t]{rl}
V\geq &\hspace{-7pt} \left|w^k\right|^2\\
f\cdot D_xV + m^B_k\cdot D_{w^k}V \leq &\hspace{-7pt}0,
\end{array}
\end{equation}
where the right-hand constraints are imposed for all $(x,w^k)\in\mathcal{B}\times\mathbb{R}^{\bn{n}{k}}$.
\item Suppose $\mathcal{B}$ is compact and $f\in C^2(\mathcal{B},\mathbb{R}^n)$. Then~\cref{eq:Sigmak C1 ineq shifted} holds with equality between the supremum and the infimum, including when Lyapunov functions are restricted to the form $V(x,w^k) = (w^k)^\mathsf{T} U(x) w^k$ with $U\in C^1(\mathcal{B},\mathbb{R}^{\bn{n}{k}\times \bn{n}{k}})$.
\end{enumerate}
\end{prop}
\begin{proof}
For the first part, assume $B$ and $V$ satisfy the constraints on the right-hand side of~\cref{eq:Sigmak C1 ineq shifted}. Letting $A(x)=Df^{[k]}(x)-BI_{\bn{n}{k}}$, the first part of \cref{thm:lyapunovfunction} applies to the system \cref{eq:shiftedaugmented} and guarantees the boundedness of $w^k(t)$ for $t\ge0$. Boundedness of $w^k(t)$ means that $B$ satisfies the constraint on the right-hand side of \cref{eq:LE sum is inf}, so $\Mu_k(x_0)\le B$ for all $x_0\in\mathcal{B}$. This implies \cref{eq:Sigmak C1 ineq shifted}, so the first part of the proposition is proved.

For the second part, let $B>\sup_{x_0\in\mathcal{B}} \Mu_k(x_0)$. By the second part of \Cref{thm:shiftedspectrumstability}, there exist $K,\lambda>0$ such that \cref{eq:exp decay} holds. Since $A\in C^1(\mathcal{B},\mathbb{R}^{\bn{n}{k}\times \bn{n}{k}})$ when $f\in C^2$, and since $\mathcal{B}$ is compact, the second part of \cref{thm:lyapunovfunction} guarantees the existence of $V(x,w^k) = (w^k)^\mathsf{T} U(x) w^k$ satisfying the constraint in \cref{eq:Sigmak C1 ineq shifted}. This is true for any $B>\sup_{x_0\in\mathcal{B}} \Mu_k(x_0)$, so~\cref{eq:Sigmak C1 ineq shifted} is an equality.
\end{proof}

To bound the Lyapunov dimension $d_L$ using the shifted spectrum approach, one option is to find a negative upper bound on $\Mu_{k+1}$ and a nonnegative upper bound on $\Mu_k$, and then use expression~\cref{eq:dimbound1} of \cref{thm:dimension_separate}. Another option, which is potentially sharper, is to bound $d_L$ directly, similarly to \cref{thm:dimcombined} in the sphere projection approach. However, whereas \cref{thm:dimcombined} avoids $y^k$ and $y^{k+1}$ by using their projections onto the unit sphere, here we retain the $y^{k+1}$ variable and shift the spectrum of $y^k$ by a factor that depends on $y^{k+1}$. We assume that $k$ is chosen so that $d_L<k+1$, as can be confirmed by a negative upper bound on $\Mu_{k+1}$. Then $y^{k+1}(t)$ is bounded and decays exponentially on average. The variable $y^k(t)$ might be unbounded, so we instead use
\begin{equation}
w(t)=\big|y^{k+1}(t)\big|^{\tfrac{B}{1-B}}y^k(t),
\label{eq:w ldim}
\end{equation}
which will be bounded for large enough $B\in[0,1)$ since $y^{k+1}(t)$ decays. The particular expression~\cref{eq:w ldim} for $w$ is chosen so that boundedness of $w$ implies $d_L\le k+B$, as asserted by \cref{lem:dL shift} below. (We do not pursue conditions under which the the bound on $d_L$ given by this lemma is an equality.)
\begin{lem}
\label{lem:dL shift}
Let $x(t)$ evolve by the ODE \cref{eq:dynsys}, where $f\in C^1(\mathcal{B},\mathbb{R}^n)$ and $\mathcal{B}$ is forward invariant. Assume each trajectory $x(t)$ in $\mathcal B$ is bounded for $t\ge0$. For any positive integer $k<n$ such that $\sup_{x_0\in\mathcal{B}}\Mu_{k+1}(x_0)<0$, let $w(t)$ be as defined by~\cref{eq:w ldim}, where $y^k(t)$ and $y^{k+1}(t)$ evolve by~\cref{eq:yk ode}. Then the global Lyapunov dimension $d_L$ over $\mathcal{B}$ is bounded above by
\begin{equation}
\label{eq:LE sum shift lemma}
d_L \le \inf_{B\in[0,1)}(k+B)\quad\text{s.t.}\quad \begin{array}[t]{l}
\text{for each } (x_0,y_0^k,y_0^{k+1})\in\mathcal{B}\times\mathbb{R}^{\bn{n}{k}}\times\mathbb{R}^{\bn{n}{k+1}},\\ ~w(t)\text{ is bounded for }t\ge0.
\end{array}
\end{equation}
\end{lem}
\begin{proof}
Let $B\in[0,1)$ be such that $w(t)$ defined by~\cref{eq:w ldim} is bounded for $t\ge0$ with each initial condition. It suffices to show that $d_L\le k+B$. Taking the logarithm of $|w(t)|$ and dividing by $t>0$ gives
\begin{equation}
\label{eq: logw}
\frac{1}{t}\log|w(t)|=\frac{B}{1-B}\frac{1}{t}\log\big|y^{k+1}(t)\big|+\frac{1}{t}\log\big|y^{k}(t)\big|.
\end{equation}
The left-hand side is nonpositive as $t\to\infty$ since $w(t)$ is bounded, so taking this limit and then maximizing over initial conditions gives
\begin{equation}
\label{eq: dL shift pf}
0\ge\left(\frac{B}{1-B}\right) \sup_{y^{k+1}_0\in\mathbb{S}^{{\bn{n}{k+1}}-1}}\limsup_{t\to\infty} \frac{1}{t}\log\big|y^{k+1}(t)\big| + \sup_{y^k_0\in\mathbb{S}^{{\bn{n}{k}}-1}}\limsup_{t\to\infty} \frac{1}{t}\log{|y^k(t)|}.
\end{equation}
Recalling expression~\cref{eq:Sigmak} for $\Mu_k$, we see that \cref{eq: dL shift pf} is equivalent to $B\left[\Mu_k(x_0)-\Mu_{k+1}(x_0)\right]-\Mu_k(x_0)\ge 0$, and so \cref{rmk:dL upper bound} gives the desired conclusion that $d_L\le k+B$ for all $x_0$.
\end{proof}
In order to apply \cref{lem:dL shift} to a particular ODE, one must show that $w(t)$ indeed remains bounded for all initial conditions. Such boundedness can be shown using the specialized Lyapunov function result of \cref{thm:lyapunovfunction}, for which we need the dynamics of $w(t)$. An ODE for $w$ is found by taking the time derivative of~\cref{eq:w ldim}, using the ODE~\cref{eq:y ode} for $y^k$ and $y^{k+1}$, and eliminating $y^k$ in favor of $w$. The resulting ODE, together with the ODEs for $x$ and $y^{k+1}$,  form the augmented system
\begin{equation}
\label{eq:dimaugmentedshifted}
        \ddt{} \begin{bmatrix}x\\w\\y^{k+1}\end{bmatrix} = \begin{bmatrix}f(x)\\p_k^B(x,w,y^{k+1})\\ Df^{[k+1]}(x)y^{k+1} \end{bmatrix},
\end{equation}
where
\begin{equation}
\label{eq:pkB def}
    p_k^B(x,w,y^{k+1}) =Df^{[k]}(x)w + \frac{B}{1-B}\frac{(y^{k+1})^\mathsf{T} Df^{[k+1]}(x)y^{k+1}}{|y^{k+1}|^2}w.
\end{equation}
\Cref{thm:lyapunovfunction} can be used to show boundedness of $w(t)$ that satisfy~\cref{eq:dimaugmentedshifted} for chosen $k$ and $B$, in which case \cref{lem:dL shift} will imply $d_L\le k+B$. The following proposition states this main result for bounding Lyapunov dimension using the shifted spectrum approach.
\begin{prop}
\label{thm:shifted dimension}
Let $x(t)$ evolve by the ODE \cref{eq:dynsys}, where $f\in C^1(\mathcal{B},\mathbb{R}^n)$ and $\mathcal{B}$ is forward invariant. Assume each trajectory $x(t)$ in $\mathcal B$ is bounded for $t\ge0$. For any positive integer $k< n$ such that $\sup_{x_0\in\mathcal{B}}\Mu_{k+1}(x_0)<0$, let $p_k^B(x,w,y^{k+1})$ be as defined by \cref{eq:pkB def} . Then the global Lyapunov dimension $d_L$ over $\mathcal{B}$ is bounded above by
\begin{equation}
\label{eq:shifted combined dimension}
d_L \leq \inf_{\substack{B\in[0,1)\\V\in C^1\hspace{6pt}}} (k+B) \quad
\text{s.t.}\quad \begin{array}[t]{rl}
V(x,w,y^{k+1}) \geq&\hspace{-7pt} \left|w\right|^2 \\
f\cdot D_x V + p_k^B \cdot D_{w} V+\big(Df^{[k+1]}y^{k+1}\big)\cdot D_{y^{k+1}} V  \leq&\hspace{-7pt} 0,
\end{array}
\end{equation}
where the right-hand constraints are imposed for all $(x,w,y^{k+1})\in\mathcal{B}\times\mathbb{R}^{\bn{n}{k}}\times\mathbb{R}^{\bn{n}{k+1}}$, and $D_xV$, $D_w V$ and $D_{y^{k+1}}V$ denote the gradients of $V(x,w,y^{k+1})$ with respect to each argument.
\end{prop}
\begin{proof}
Let $B$ and $V(x,w,y^{k+1})$ satisfy the constraints in~\cref{eq:shifted combined dimension}. It suffices to show that $d_L\le k+B$. The inequality constraints are exactly the Lyapunov function conditions~\cref{eq:Lyapunov function conditions} in \cref{thm:lyapunovfunction}, written for the case where the nonlinear variables are $(x,y^{k+1})$ and evolve as in \cref{eq:dimaugmentedshifted}, and the linear variable $w$ evolves as $\ddt{}w=A(x,y^{k+1})w$ with the form of $A$ implied by  \cref{eq:pkB def}. Thus, by the first part of \cref{thm:lyapunovfunction}, the constraints in \cref{eq:shifted combined dimension} imply that $w(t)$ is bounded for all initial conditions $(x_0,w_0,y_0^{k+1})$. \Cref{lem:dL shift} then implies that $d_L\le k+B$, which completes the proof.
\end{proof}

\subsection{Sum-of-squares relaxations}
\label{sec:sos}

The upper bounds on LE sums and Lyapunov dimension formulated in \cref{sec:sphere,sec:shifted} are stated as, or can be reformulated as, minimizations subject to pointwise inequalities on $V$. When the ODE right-hand side $f(x)$ is polynomial, and when we restrict to auxiliary functions $V$ in some finite-dimensional space of polynomials, then the pointwise inequalities can be enforced using SOS conditions. This gives SOS optimization problems that are computationally tractable when the polynomial degrees are not too large, and whose minima give upper bounds on the desired quantities.  For the shifted spectrum approach of \cref{sec:shifted}, the SOS relaxations are immediate after choosing a polynomial space for $V$ and enforcing nonnegativity on some set $\Omega$ via membership in $\Sigma^\Omega$. For the sphere projection approach of \cref{sec:sphere}, the relevant minmax problems can be expressed as constrained minimizations, after which the SOS relaxations are simple. We review the necessary basics of SOS polynomials here.

In order for a polynomial $S:\mathbb{R}^n\to\mathbb{R}$ to satisfy $S(x)\ge0$ for all $X\in\mathbb{R}^n$, it suffices that $S\in \Sigma$, where $\Sigma$ denotes the set of polynomials that are representable as a sum of squares of other polynomials. The set of nonnegative polynomials strictly contains $\Sigma$, and the relationship between these sets has been extensively studied \citep{reznick2000some}. Often one wants to enforce $S(x)\ge0$ on a subset of $\mathbb{R}^n$ without requiring nonnegativity on all of $\mathbb{R}^n$. For instance, one may be interested in a set $\Omega$ with a `semiagebraic' definition in terms of polynomial equalities and inequalities,
\begin{equation}
\label{eq:semialg set}
\Omega = \left\{ X\in\mathbb{R}^n~:~g_i(X)\ge0\text{ for }i=1,\ldots,I,~h_j(X)=0\text{ for }j=1,\ldots,J \right\},
\end{equation}
where the $g_i\in\mathbb{R}[X]$ and $h_i\in\mathbb{R}[X]$ are given. Here $\mathbb{R}[X]$ denotes the ring of polynomials in $X$ with real coefficients, and $\mathbb{R}^n[X]$ denotes the vectors of such polynomials taking values in $\mathbb{R}^n$. There are various ways to define a set $\Sigma^\Omega$ of polynomials whose nonnegativity on $\Omega$ is implied by SOS conditions. Here we choose $\Sigma^\Omega$ to be the quadratic module generated by the polynomials that define~$\Omega$,
\begin{equation}
\label{eq:quad module}
\Sigma^\Omega = \left\{ \sigma_0+\sum\limits_{i=1}^I \sigma_ig_i+\sum\limits_{j=1}^J\rho_jh_j~:~\sigma_i\in\Sigma\text{ for }i=0,\ldots,I,~\rho_j\in\mathbb{R}[X]\text{ for }j=1,\ldots,J\right\}.
\end{equation}
Any polynomial in $\Sigma^\Omega$ must be nonnegative on $\Omega$ because $\sigma_0(X)\ge0$ at every $X\in\mathbb{R}^n$, while each $\sigma_i(X)g_i(X)\ge0$ and each $\rho_j(X)h_j(X)=0$  at every $X\in\Omega$. When $\Omega=\mathbb{R}^n$, there are no $g_i$ or $\rho_j$, so $\Sigma^\Omega=\Sigma$. In the preceding paper \cite{oeri2023convex}, the condition $S\in\Sigma^\Omega$ was expressed in a different but equivalent way: by saying there exist $\sigma_i$ and $\rho_j$ such that $S-\sum\sigma_ig_i-\sum\rho_jh_j$ is SOS.

Computational formulations enforce membership not only in $\Sigma^\Omega$ but in a finite-dimensional subset. The simplest choice is to enforce membership in~$\Sigma^\Omega_\nu$, which is the truncation of the quadratic module~\cref{eq:quad module} where each $\sigma_i$ and $\rho_j$ has polynomial degree no larger than $\nu$, for some $\nu$ in the set $\mathbb{Z}_\ge$ of nonnegative integers. This gives a hierarchy of computational formulations, in which raising $\nu$ cannot worsen results---and typically improves them---but makes computations more expensive. In practice, it is often useful to choose the spaces for $\sigma_i$ and $\rho_i$ more carefully. In our computations of \cref{sec:examples}, for instance, we impose symmetries and allow different maximum degrees in different variables. For simplicity in the present subsection, however, we write formulations in terms of the degree-$\nu$ truncation~$\Sigma^\Omega_\nu$.

\subsubsection{Sphere projection}
\label{sec:sphere sos}

The sphere projection approach of \cref{sec:sphere} is based on the general upper bound~\cref{eq:Phi-opt} on time averages, which can be expressed as
\begin{equation}
\max_{X_0\in \Omega}\overline{\Phi} 
\le \inf_{V\in C^1(\Omega)} B\quad \text{s.t.}\quad B- \Phi - F\cdot DV \ge0~\;\forall X\in\Omega 
\label{eq:X relax 1}.
\end{equation}
An SOS relaxation of the right-hand side gives an upper bound, stated in the first part of \cref{thm:TGD-sos} below. Furthermore, this upper bound is an equality if $\Omega$ is forward invariant and its semialgebraic specification~\cref{eq:semialg set} has the Archimedean property, as stated in the second part of \cref{thm:TGD-sos} below. The Archimedean property implies compactness of $\Omega$ and is only slightly stronger; any semialgebraic specification of compact $\Omega$ that is not Archimedean can be given this property by adding the polynomial $g_i(X)=R^2-|X|^2$ with a constant $R$ that is large enough to not change $\Omega$ itself. A proof of the second part of \cref{thm:TGD-sos} can be found in \cite[Theorem 1]{lakshmi2020finding}. Briefly, since $\Omega$ is compact and forward invariant, equality in~\cref{eq:X relax 1} is given by the main theorem of~\cite{tobasco2018optimal}, whose applicability on compact manifolds is addressed after~(8) in \cite{oeri2023convex}. Then, for each suboptimal $B$, there exists a polynomial $V$ such that $B-\Phi - F\cdot DV>0$ on $\Omega$, and Putinar's Positivstellensatz \cite[Lemma 4.1]{putinar1993positive} guarantees that $B-\Phi - F\cdot DV$ belongs to $\Sigma^\Omega$.

\begin{prop}
\label{thm:TGD-sos}
Let $X(t)$ evolve by the ODE \cref{eq:ODE}, where $F\in\mathbb{R}^n[X]$. Let  $\Omega\subset\mathbb{R}^n$ be a semialgebraic set that is forward invariant, and whose quadratic module is denoted $\Sigma^\Omega$. Let $\Phi\in\mathbb{R}[X]$.
\begin{enumerate}
\item If each trajectory $X(t)$ in $\Omega$ is bounded for $t\ge0$, then
\begin{equation}
\label{eq:X relax 2}
\sup_{X_0\in \Omega}\overline{\Phi} 
\le \inf_{V\in \mathbb{R}[X]} B\quad \text{s.t.}\quad B- \Phi - F\cdot DV \in \Sigma^\Omega. 
\end{equation}
\item If $\Omega$ is compact and its semialgebraic specification is Archimedean, then
\begin{equation}
\label{eq:X relax sharp}
\max_{X_0\in \Omega}\overline{\Phi} 
= \inf_{V\in \mathbb{R}[X]} B\quad \text{s.t.}\quad B- \Phi - F\cdot DV \in \Sigma^\Omega. 
\end{equation}
\end{enumerate}
\end{prop}

All of the upper bounds from \cref{sec:sphere} can be relaxed into SOS optimization problems by the same reasoning that leads from~\cref{eq:Phi-opt} to~\cref{eq:X relax 1} to~\cref{eq:X relax 2}. \Cref{thm:projectionmethod sos} below is an SOS relaxation of \cref{thm:projectionmethod} for bounding LE sums. Computations giving such bounds also imply bounds on Lyapunov dimension via \cref{thm:dimension_separate}, but potentially sharper bounds on dimension can be computed directly using \cref{thm:dimcombined sos}, which is an SOS relaxation of \cref{thm:dimcombined}. In the proposition statements, the domain $\mathcal{B}$ is assumed to have the semialgebraic specification
\begin{equation}
\label{eq:calB}
\mathcal{B}=\left\{ x\in\mathbb{R}^n~:~g_i(x)\ge0\text{ for }i=1,\ldots,I,~h_j(x)=0\text{ for }j=1,\ldots,J \right\}.
\end{equation}
This is understood to include the $\mathcal{B}=\mathbb{R}^n$ case when there are no $g_i$ or $h_j$. The $k=1$ special case of \cref{thm:projectionmethod sos} is essentially proposition~2 of \cite{oeri2023convex}, although the SOS conditions are more explicit in~\cite{oeri2023convex}, and the restriction to finite-dimensional polynomial spaces is more explicit here.

\begin{rmk}
In the following \cref{thm:projectionmethod sos,thm:dimcombined sos,thm:shiftedmethod sos,thm:shifted dimension sos}, for fixed $k\in\{1,\ldots,n\}$ and polynomial degrees $\nu$ and $\nu'$, the resulting upper bound $C_i(k,\nu,\nu')$ can be computed by solving SOS programs using standard software (see \cref{sec:examples}). For certain combinations $(k,\nu,\nu')$, there will not exist any $V$ in the chosen space of polynomials that satisfy the SOS constraints, in which case we understand the infimum over bounds to be $+\infty$, meaning that we obtain no upper bound. In \cref{thm:projectionmethod sos,thm:dimcombined sos}, the values $C_1$ and $C_2$ are found by solving a single SOS program in which the bound $B$ is minimized. In \cref{thm:shiftedmethod sos,thm:shifted dimension sos}, the value $B$ must be fixed in each SOS program, which determines feasibility or infeasibility of the constraints with this $B$ value. In these cases the minimum feasible $B$ is found by repeatedly solving SOS programs to find $C_3$ and $C_4$.
\end{rmk}

\begin{prop}
\label{thm:projectionmethod sos}
Let $x(t)$ evolve by the ODE \cref{eq:dynsys}, where $f\in\mathbb{R}^n[X]$ and  $\mathcal{B}\subset\mathbb{R}^n$ is a semialgebraic set that is forward invariant. Let $\Sigma^\Omega$ be the quadratic module~\cref{eq:quad module} for the set $\Omega=\mathcal{B}\times\mathbb{S}^{\bn{n}{k}-1}$ whose semialgebraic specification includes the polynomial constraints on $x$ in~\cref{eq:calB} and the equality $\big| z^k\big|-1=0$. Let $ \Phi_k(x,z^k)$ and $\ell_k(x,z^k)$ be as defined in \cref{eq:Phik} and \cref{eq:ell}, respectively. 
\begin{enumerate}
\item Suppose each trajectory $x(t)$ in $\mathcal{B}$ is bounded for $t\ge0$. For each $k\in\{1,\ldots,n\}$, and each pair of polynomial degrees $(\nu,\nu')$, the maximum leading LE sum $\Mu_k$ among trajectories in $\mathcal{B}$ is bounded above by
\begin{equation}
\label{eq:Sigmak sos ineq}
\sup_{x_0\in \mathcal{B}} \Mu_k(x_0) \leq C_1(k,\nu,\nu'),
\end{equation}
where
\begin{equation}
\label{eq:C1 def}
C_1(k,\nu,\nu')=\inf_{V\in \mathbb{R}[x,z^k]_\nu}B\quad\text{s.t.}\quad B-\Phi_k - f\cdot D_xV-\ell_k\cdot D_{z^k}V\in\Sigma^\Omega_{\nu'}.
\end{equation}
    \item Suppose $\mathcal{B}$ is compact and its semialgebraic specification is Archimedean. Then, for each $k$,
\begin{equation}
\label{eq:Sigmak sos eq}
\sup_{x_0\in \mathcal{B}} \Mu_k(x_0) = \sup_{\nu,\nu'\in\mathbb{Z}_\ge} C_1(k,\nu,\nu').
\end{equation}
\end{enumerate}
\end{prop}
\begin{proof}
The first part is an SOS relaxation of the upper bound~\cref{eq:Sigmak C1 ineq} on $\Mu_k$ for each $k$. Equivalently, it is a particular case of~\cref{eq:X relax 2} with $X$, $F$ and $\Phi$ as described above in \cref{thm:projectionmethod}, where $V$ is further constrained by having degree no larger than $\nu$, and where $\sigma_i$ and $\rho_j$ in the quadratic module representation~\cref{eq:quad module} have degrees no larger than $\nu'$. The second part is the analogous special case of~\cref{eq:X relax sharp}. This equality is applicable because $\mathcal{B}$ is specified as an Archimedean subset of $\mathbb{R}^n$, and so adding $|z^k|-1=0$ gives an Archimedean specification of $\Omega$ in $\mathbb{R}^n\times\mathbb{R}^{\bn{n}{k}}$. All polynomial degrees are admitted in~\cref{eq:X relax sharp}, which is equivalent to the supremum over $\nu$ and $\nu'$ in~\cref{eq:Sigmak sos eq}.
\end{proof}
\begin{prop}
\label{thm:dimcombined sos}
Let $x(t)$ evolve by the ODE \cref{eq:dynsys}, where $f\in\mathbb{R}^n[X]$ and  $\mathcal{B}\subset\mathbb{R}^n$ is a semialgebraic set that is forward invariant. Let the positive integer $k< n$ be such that $\sup_{x_0\in\mathcal{B}}\Mu_{k+1}(x_0)<0$. Let $\Sigma^\Omega$ be the quadratic module~\cref{eq:quad module} for the set $\Omega=\mathcal{B}\times\mathbb{S}^{\bn{n}{k}-1}\times\mathbb{S}^{\bn{n}{k+1}-1}$, whose semialgebraic specification includes the polynomial constraints on $x$ in~\cref{eq:calB} and the equalities $\big| z^k\big|-1=0$ and $\big| z^{k+1}\big|-1=0$. Let $ \Phi_k(x,z^k)$ and $\ell_k(x,z^k)$ be defined as in \cref{eq:Phik} and \cref{eq:ell}, respectively, and likewise for  $ \Phi_{k+1}(x,z^{k+1})$ and $\ell_{k+1}(x,z^{k+1})$.
\begin{enumerate}
\item Suppose each trajectory $x(t)$ in $\mathcal{B}$ is bounded for $t\ge0$. For each pair of polynomial degrees $(\nu,\nu')$, the maximum Lyapunov dimension $d_L$ among trajectories in $\mathcal{B}$ is bounded above by
\begin{equation}
\label{eq:C2 ineq}
d_L\le k + C_2(k,\nu,\nu',0)
\end{equation}
where
\begin{equation}
\label{eq:C2 def}
C_2(k,\nu,\nu',\varepsilon)=\hspace{-6pt}\inf\limits_{\substack{B\in[0,1]\\V\in \mathbb{R}[x,z^k,z^{k+1}]_\nu}}\hspace{-10pt}B
\quad \text{s.t.}\quad
\varepsilon-[\Psi^B + f\cdot D_xV+\ell_k\cdot D_{z^k}V +\ell_{k+1}\cdot D_{z^{k+1}}V] \in \Sigma^\Omega_{\nu'},
\end{equation}
with $C_2=+\infty$ if the constraint set is empty.
\item Suppose $\mathcal{B}$ is compact and its semialgebraic specification is Archimedean. Let $j$ be the minimum admissible $k$, as in \cref{def:ldim} for $d_L$. Then,
\begin{equation}
\label{eq:eq:dL eq sphere}
d_L= j~ +\inf_{\substack{\nu,\nu'\in\mathbb{Z}{_\ge}\\\varepsilon\ge0}}
C_2(j,\nu,\nu',\varepsilon).
\end{equation}
\end{enumerate}
\end{prop}
\begin{proof}
The first part follows from the upper bound~\cref{eq:dL bounds 1} on $d_L$ because the SOS constraint in~\cref{eq:C2 def} with $\varepsilon=0$ implies the nonpositivity constraint on the right-hand side of~\cref{eq:dL bounds 1}. For the second part, we first establish upper and lower bounds on $C_2$ in terms of $d_L$. For the lower bound, note that the constraint in~\cref{eq:C2 def} implies the constraint in the definition~\cref{eq:Beps} of $B_\varepsilon$. This implies $B_\varepsilon\le C_2(j,\nu,\nu',\varepsilon)$, with which the first inequality in~\cref{eq:dL bounds 2} gives 
\begin{equation}
\label{eq:C2 lower bound}
d_L~-~\varepsilon\,\left|\sup_{x_0\in\mathcal{B}}\Mu_{j+1}(x_0)\right|^{-1}
\le j+C_2(j,\nu,\nu',\varepsilon)
\end{equation}
for all $\varepsilon>0$ and $\nu,\nu'\in\mathbb{Z}_\ge$. Minimizing over these parameters gives
\begin{equation}
\label{eq:C2 lower bound 2}
d_L\le j~ +\inf_{\substack{\nu,\nu'\in\mathbb{Z}{_\ge}\\\varepsilon>0}}
C_2(j,\nu,\nu',\varepsilon),
\end{equation}
and this inequality remains true when the constraints allow $\varepsilon=0$, due to~\cref{eq:C2 ineq}.

For the upper bound on $C_2$, note that since $\mathcal{B}$ has an Archimedean specification, so does $\Omega$. Thus we can apply the general SOS equality~\cref{eq:X relax sharp}, in particular with $F(X)$ being the right-hand side of the ODE system~\cref{eq:augmenteddim1}, with $\Phi(X)$ being $\Psi^B$, and fixing $k=j$. This gives
\begin{equation}
\label{eq:dL sphere sos pf 1}
\sup_{\substack{x_0\in\mathcal{B}~~~~~\\z^j_0\in\mathbb{S}^{\bn{n}{j}-1}\\z^{j+1}_0\in\mathbb{S}^{\bn{n}{j+1}-1}}}
\hspace{-14pt}\overline{\Psi}^B =
\inf_{V\in \mathbb{R}[x,z^j,z^{j+1}]}\hspace{-12pt}\varepsilon
\quad\text{s.t.}\quad
\varepsilon-\left[\Psi^B + f\cdot D_xV+\ell_j\cdot D_{z^j}V +\ell_{j+1}\cdot D_{z^{j+1}}V\right]\in\Sigma^\Omega.
\end{equation}
Note that the variable called $B$ in~\cref{eq:X relax sharp} is $\varepsilon$ in~\cref{eq:dL sphere sos pf 1}, and here $B$ has a different meaning. Let $B$ be such that the constraint holds in expression~\cref{eq:dL sphere pf 2} for $d_L$. This constraint is equivalent to both sides of~\cref{eq:dL sphere sos pf 1} being nonpositive, and nonpositivity of the right-hand infimum in~\cref{eq:dL sphere sos pf 1} implies that there exists polynomial $V$ satisfying the SOS constraint for any $\varepsilon>0$. Since such $V$ exists for any $B>d_L-j$ , 
\begin{equation}
d_L-j~\ge\inf_{\substack{B\in[0,1]\\V\in \mathbb{R}[x,z^j,z^{j+1}]}}\hspace{-8pt}B
\quad\text{s.t.}\quad
\varepsilon-\left[\Psi^B + f\cdot D_xV+\ell_j\cdot D_{z^j}V +\ell_{j+1}\cdot D_{z^{j+1}}V\right]\in\Sigma^\Omega
\end{equation}
for all $\varepsilon>0$. Restricting the right-hand minimization to $V$ of degree $\nu$ and to $\sigma_i$ and $\rho_j$ of degree $\nu'$ gives the definition of $C_2(j,\nu,\nu',\varepsilon)$. Thus, for all $\varepsilon>0$,
\begin{equation}
\label{eq:C2 upper bound}
j~+\inf_{\nu,\nu'\in\mathbb{Z{_\ge}}}C_2(j,\nu,\nu',\varepsilon) \le d_L.
\end{equation}
(This inequality need not hold with any finite $\nu$ and $\nu'$.) The upper bound~\cref{eq:C2 upper bound} on $C_2$, combined with the lower bound~\cref{eq:C2 lower bound 2}, establishes the claim~\cref{eq:eq:dL eq sphere} of the second part and completes the proof.
\end{proof}

For each fixed $(k,\nu,\nu')$, the minimization problems defining $C_1$ in~\cref{eq:C1 def} and $C_2$ in~\cref{eq:C2 def} are SOS programs. That is, polynomials of finite degree are constrained to be SOS, and these polynomials, as well as the optimization objective, are affine in the optimization parameters. Here the optimization parameters are $B$ and the coefficients of $V$ (which depend on the basis in which $V$ is represented). Such SOS programs are computationally tractable so long as the number of variables and the polynomial degrees in the SOS constraints are not too large. Increasing $\nu$ and $\nu'$ gives a hierarchy of SOS programs with increasing computational cost whose minima cannot increase. Typically these minima decrease as $\nu$ and $\nu'$ are raised, and the second part of each proposition guarantees convergence to the sharp bound under mild conditions.

\subsubsection{Shifted spectrum}
\label{sec:shifted sos}

Now we give the SOS formulations for the shifted spectrum approach. \Cref{thm:shiftedmethod sos} below is an SOS relaxation of \cref{thm:shiftedmethod} for bounding LE sums. 
Computations giving such bounds also imply bounds on Lyapunov dimension via \cref{thm:dimension_separate}, but potentially sharper bounds on dimension can be computed directly using \cref{thm:shifted dimension sos}, which is an SOS relaxation of \cref{thm:shifted dimension}. We do not prove that \cref{thm:shifted dimension sos} gives sharp results in general, as we have for our other methods. We also omit its short proof that, as in the first part of \cref{thm:shiftedmethod sos}, simply strengthens nonnegativity conditions into SOS conditions.

\begin{prop}
\label{thm:shiftedmethod sos}
Let $x(t)$ evolve by the ODE \cref{eq:dynsys}, where $f\in\mathbb{R}^n[X]$ and  $\mathcal{B}\subset\mathbb{R}^n$ is a semialgebraic set that is forward invariant. Let $\Sigma^\Omega$ be the quadratic module~\cref{eq:quad module} for the set $\Omega=\mathcal{B}\times\mathbb{R}^{\bn{n}{k}}$, using the same constraints~\cref{eq:calB} that specify $\mathcal{B}$. Let $m^B_k(x,w^k)$ be as defined by \cref{eq:ellB}.
\begin{enumerate}
\item Suppose each trajectory $x(t)$ in $\mathcal{B}$ is bounded for $t\ge0$. For each $k\in\{1,\ldots,n\}$, and each pair of polynomial degrees $(\nu,\nu')$, the maximum leading LE sum $\Mu_k$ among trajectories in $\mathcal{B}$ is bounded above by
\begin{equation}
\label{eq:C3 ineq}
\sup_{x_0\in \mathcal{B}} \Mu_k(x_0) \leq C_3(k,\nu,\nu'),
\end{equation}
where
\begin{equation}
\label{eq:C3 def}
C_3(k,\nu,\nu') = \inf_{V\in \mathbb{R}[x,w^k]_\nu}B
\quad
\text{s.t.}\quad \begin{array}[t]{r}
V-|w^k|^2\in\Sigma^\Omega_{\nu'}\phantom{.} \\
-[f\cdot D_xV+m_k^B\cdot D_{w^k}V]\in\Sigma^\Omega_{\nu'}.
\end{array}
\end{equation}
\item Suppose $\mathcal{B}$ is compact and its semialgebraic specification is Archimedean. Then, for each $k$,
\begin{equation}
\label{eq:C3 eq}
\sup_{x_0\in \mathcal{B}} \Mu_k(x_0) = \inf_{\nu,\nu'\in\mathbb{Z}_\ge} C_3(k,\nu,\nu'),
\end{equation}
where the infimum in \cref{eq:C3 def} is restricted to $V(x,w^k) = (w^k)^\mathsf{T} U(x) w^k$ with symmetric polynomial matrix $U\in \mathbb{R}^{\bn{n}{k}\times \bn{n}{k}}[x]$, and where $\sigma_i(x,w^k)$ and $\rho_j(x,w^k)$ in the SOS representation~\cref{eq:quad module} are likewise quadratic in~$w^k$.
\end{enumerate}
\end{prop}
\begin{proof}
The first part follows from the upper bound~\cref{eq:Sigmak C1 ineq shifted} on $\Mu_k$ because the SOS constraints in~\cref{eq:C3 def} imply the nonnegativity constraints in~\cref{eq:Sigmak C1 ineq shifted}. For the second part, let $B_1$ and $B_2$ be arbitrary with $B_1>B_2>\sup_{x_0\in \mathcal{B}} \Mu_k(x_0)$. It suffices to show that there exists a polynomial matrix $U(x)$ such that $V(x,w^k) = (w^k)^\mathsf{T} U(x) w^k$ satisfies the SOS constraints of~\cref{eq:C3 def} for $B=B_1$ with no restrictions on polynomial degree. 

The second part of~\cref{thm:shiftedmethod} guarantees the existence of a symmetric $U\in C^1(\mathbb{R}^n,\mathbb{R}^{\bn{n}{k}\times \bn{n}{k}})$ for which $V=(w^k)^\mathsf{T} U w^k$ satisfies the pointwise inequalities in~\cref{eq:Sigmak C1 ineq shifted} with $B=B_2$. Equivalently, this $U$ satisfies $(w^k)^\mathsf{T}Z_i(x)w^k\ge0$ for all $x\in\mathcal{B}$ and
\begin{align}
Z_1(x) &= U(x)-I_{\bn{n}{k}},\\
Z_2(x) &= -f(x)\cdot D_xU(x)+2\,U(x)[B\,I_{\bn{n}{k}}-Df^{[k]}(x)],
\end{align}
where $B=B_2$ and $I_{\bn{n}{k}}$ denotes the $\bn{n}{k}\times \bn{n}{k}$ identity matrix. This, in turn, is equivalent to both $Z_i(x)$ being positive semidefinite for all $x\in\mathcal{B}$. We instead want both $Z_i(x)$ to be strictly positive definite on $\mathcal{B}$, which is achieved for $Z_1$ by scaling $U$ larger and is then achieved for $Z_2$ by taking $B=B_1$.

Because there exists $U\in C^1$ such that both $Z_i(x)$ are positive definite on $\mathcal{B}$, and $\mathcal{B}$ is compact, there also exists a polynomial matrix $U\in\mathbb{R}^{\bn{n}{k}\times \bn{n}{k}}[x]$ for which both $Z_i(x)$ are positive definite. This follows from a strengthened version of the Stone-Weierstrass approximation theorem~\cite{Lorentz1986approximation}. This polynomial $U$ satisfies the SOS constraints of~\cref{eq:C3 def} due to $\mathcal{B}$ being Archimedean and $Z_i$ being positive definite on~$\mathcal{B}$. Recalling that $\mathcal B$ is the set where all $g_i(x)\ge0$ and $h_j(x)=0$, applying a matrix Positivstellensatz \cite[theorem 2]{scherer2006matrix} guarantees existence of polynomial matrices $S_i(x)$ and $T_i(x)$, with the $S_i(x)$ being squares of other polynomial matrices, such that
\begin{equation}
Z_1(x)=S_0(x)+\sum_{i=1}^Ig_i(x)S_i(x)+\sum_{j=1}^Jh_j(x)T_j(x),
\end{equation}
and likewise for $Z_2$. The existence of such representations implies that both $(w^k)^\mathsf{T}Z_iw^k$ belong to the quadratic module $\Sigma^\Omega$. Therefore $V=(w^k)^\mathsf{T}Uw^k$ satisfies the constraints of~\cref{eq:C3 def} with $B=B_1$, and the second part of the proposition is proved.
\end{proof}
\begin{prop}
\label{thm:shifted dimension sos}
Let $x(t)$ evolve by the ODE \cref{eq:dynsys}, where $f\in\mathbb{R}^n[X]$ and  $\mathcal{B}\subset\mathbb{R}^n$ is a semialgebraic set that is forward invariant. Assume each trajectory $x(t)$ in $\mathcal{B}$ is bounded for $t\ge0$. Let the positive integer $k< n$ be such that $\sup_{x_0\in\mathcal{B}}\Mu_{k+1}(x_0)<0$. Let $\Sigma^\Omega$ be the quadratic module~\cref{eq:quad module} for the set $\Omega=\mathcal{B}\times\mathbb{R}^{\bn{n}{k+1}}\times\mathbb{R}^{\bn{n}{k}}$, using the same constraints~\cref{eq:calB} that specify $\mathcal{B}$. For each $k\in\{1,\ldots,n\}$, and each pair of polynomial degrees $(\nu,\nu')$, the maximum Lyapunov dimension $d_L$ among trajectories in $\mathcal{B}$ is bounded above by
\begin{equation}
d_L\le k + C_4(k,\nu,\nu'),
\end{equation}
where
\begin{equation}
\label{eq:C4 def}
C_4(k,\nu,\nu')=\inf_{\substack{B\in [0,1]\\V\in \mathbb{R}[x,y^{k+1},w]_\nu}}B\quad\text{s.t.}\quad 
\begin{array}[t]{r}
V-\left|w\right|^2\in\Sigma^\Omega_{\nu'}\phantom{,} \\
(1-B)\left| y^{k+1}\right|^2
p_k^B \in\Sigma^\Omega_{\nu'},
\end{array}
\end{equation}
with $C_4=+\infty$ if the constraint set is empty.
\end{prop}

\subsection{Symmetry exploitation}
\label{sec:syms}
If an ODE being studied using auxiliary function methods has symmetry, then in many cases the same symmetry can be imposed on the auxiliary functions, leading to symmetries in SOS expressions. Symmetries in SOS programs can be exploited computationally by block diagonalizing the corresponding semidefinite programs (SDPs) \citep{gatermann2004symmetry}. An ODE system~\cref{eq:ODE} with right-hand side $F(X)$ is said to be invariant under a map $\Lambda:\mathbb{R}^n\to\mathbb{R}^n$ if $F$ is $\Lambda$-equivariant, meaning that $F(\Lambda X) = \Lambda F(X)$ for all $X$. In such cases, $X(t)$ solves the ODE if and only if $\Lambda X(t)$ does. 
For any ODE, the set of such transformations forms a group. We focus on the common situation where each $\Lambda$ is an orthogonal linear transformation on $\mathbb{R}^n$, meaning the group of such $\Lambda$ is a subgroup of the orthogonal group $O(n)$. (In some contexts the set of $\Lambda$ would be called the linear representation of a group, rather than a group itself, but we do not draw this distinction.) In many formulations of auxiliary function methods for ODEs, if the ODE is invariant under $\Lambda$ then the results are provably unchanged if the optimization is over the restricted set of $V$ which are $\Lambda$-invariant, meaning $V(X)=V(\Lambda X)$ for all $X$. Such a result in the context of bounding time averages, which is proved in the appendix of \cite{oeri2023convex}, justifies imposing symmetries on $V$ in all of our sphere projection methods. An analogous argument holds for the Lyapunov function conditions in \cref{thm:lyapunovfunction}, and this justifies imposing symmetries on $V$ in all of our shifted spectrum methods.

The particular symmetries that can be imposed on $V$, and on the polynomials $\sigma_i$ and $\rho_j$ appearing in Positivstellans\"atze, follow from the symmetries of ODEs on augmented state space that underlie our various formulations. First, regardless of whether the $x$ ODE has any symmetry, the augmented ODEs have a sign-change symmetry due to linearity of the $y^k$ dynamics. These sign-change symmetries are summarized by \cref{thm:signsymm}. Second, if the right-hand side $f(x)$ of the $x$ ODE is equivariant under an orthogonal transformation $\Lambda$, this induces corresponding equivariance on the right-hand sides of the augmented ODEs, as summarized by \cref{thm:generalsymm}. In the proposition, $\Lambda^{(k)}$ denotes the $k^\mathrm{th}$ multiplicative compound of $\Lambda$ (see\ \cref{sec:compounds}), $I_n$ denotes the $n\times n$ identity matrix, and $\diag(A,B,\dots)$ denotes a block diagonal matrix with square blocks $A,B,\ldots$.
\begin{prop}
\label{thm:signsymm}
    Let $\ddt{}x(t)=f(x(t))$ with $f:\mathcal{B}\to\mathbb{R}^n$. The augmented ODEs are invariant under the following orthogonal transformations.
    \begin{enumerate}
        \item The ODE system \cref{eq:projectedaugmented} is invariant under $\diag(I_n, -I_{\bn{n}{k}})$.
        \item The ODE system \cref{eq:augmenteddim1} is invariant under $\diag(I_n, \pm I_{\bn{n}{k}}, \mp I_{\bn{n}{k+1}})$.
        \item The ODE system \cref{eq:shiftedaugmented} is invariant under $\diag(I_n, -I_{\bn{n}{k}})$.
        \item The ODE system \cref{eq:dimaugmentedshifted} is invariant under $\diag(I_n,\pm I_{\bn{n}{k}},\mp I_{\bn{n}{k+1}})$.
    \end{enumerate}
\end{prop}
\begin{proof}
    All four parts have very similar proofs, so we prove only the last. The claim amounts to the right-hand side of \cref{eq:dimaugmentedshifted} being equivariant under negation of $y^{k+1}$. For $f(x)$ this statement is trivial. The requirement that $Df^{[k+1]}(x)(-y^{k+1})=-Df^{[k+1]}(x)y^{k+1}$ is also immediate. The final requirement that $p_k^B(x,w,-y^{k+1}) = p_k^B(x,w,y^{k+1})$ holds because $y^{k+1}$ appears quadratically in both the numerator and denominator in \cref{eq:pkB def}.
\end{proof}
\begin{prop}
\label{thm:generalsymm}
    Let $\mathcal{B}\subset\mathbb{R}^n$ be invariant under $\Lambda\in O(n)$. If $\ddt{}x(t)=f(x(t))$ on $\mathcal{B}$ is invariant under $\Lambda$, then the augmented ODEs are invariant under the following orthogonal transformations.
    \begin{enumerate}
        \item The ODE system \cref{eq:projectedaugmented} is invariant under $\diag(\Lambda,\Lambda^{(k)})$. 
        \item The ODE system \cref{eq:augmenteddim1} is invariant under $\diag(\Lambda, \Lambda^{(k)}, \Lambda^{(k+1)})$. 
        \item The ODE system \cref{eq:shiftedaugmented} is invariant under $\diag(\Lambda, \Lambda^{(k)})$. 
        \item The ODE system \cref{eq:dimaugmentedshifted} is invariant under $\diag(\Lambda, \Lambda^{(k)}, \Lambda^{(k+1)})$. 
    \end{enumerate}
\end{prop}
\begin{proof}
By assumption $f(x)$ is equivariant under $\Lambda$, meaning $f(\Lambda x)=\Lambda f(x)$. Orthogonality of $\Lambda$ implies orthogonality of the multiplicative compound $\Lambda^{(k)}$ by \cref{thm:multcompoundprops}, so all transformations in the proposition are orthogonal. The claim is proved by showing that the right-hand sides of the augmented ODEs are equivariant under the claimed transformations. All four parts are very similar, so we prove only the first. For this we must show that $ \ell_k(\Lambda x, \Lambda^{(k)} z^k) = \Lambda^{(k)}\ell_k(x,z^k)$. It is sufficient to show that
\begin{equation}
    Df^{[k]}(\Lambda x) \Lambda^{(k)}z^k  =  \Lambda^{(k)}Df^{[k]}(x) z^k,
    \label{eq:proofeq}
\end{equation}
so that
\begin{equation}
     \Phi_k(\Lambda x,\Lambda^{(k)}z^k) = \Phi_k(x,z^k)
\end{equation}
by orthogonality of $\Lambda^{(k)}$. Since $Df (\Lambda x) = \Lambda Df(x)\Lambda^\mathsf{T}$, taking the $k^\mathrm{th}$ additive compound of this equation gives, via \cref{thm:addcompoundprops},
\begin{equation}
Df^{[k]}(\Lambda x) = \Lambda^{(k)} Df^{[k]}(x)\left(\Lambda^{(k)}\right)^\mathsf{T}.
\end{equation}
Orthogonality then gives \cref{eq:proofeq}. The proofs of the other parts follow similarly, noting that \cref{eq:proofeq} remains valid with $y^k$ or $w^k$ in place of $z^k$.
\end{proof}

Combining both propositions gives a finite symmetry group for the augmented systems. For example, if $\mathcal{B}$ is invariant and $f$ is equivariant under a group generated by $\{\Lambda_1, \dots, \Lambda_m\}$ then the right-hand sides of the augmented systems \cref{eq:projectedaugmented} and \cref{eq:shiftedaugmented} are equivariant under the group generated by
\begin{equation}
\label{eq:augmented symmetries}
    \left\{\diag(\Lambda_1, \Lambda_1^{(k)}),\dots,\diag(\Lambda_m, \Lambda_m^{(k)}),\diag(I_n, -I_{\bn{n}{k}})\right\}.
\end{equation}
Thus we can restrict to $V$ that are invariant under the same transformations, and likewise for any $\sigma_i$ and $\rho_i$ appearing in quadratic module representations. The examples of the next section illustrate such symmetry exploitation for particular ODEs.

\section{Computational examples}
\label{sec:examples}
We now present computational examples in which our bounding formulations that use SOS optimization are applied to particular ODEs. It has been conjectured that, for all systems with equilibria embedded in chaotic attractors, the supremum in the formula~\cref{eq:dL} for Lyapunov dimension is attained on an equilibrium \citep{eden1989abstract, leonov1993eden}. This indeed is the case for many simple systems, including Lorenz's 1963 model \citep{leonov2016lyapunov}. In order to demonstrate the power of our methods in general, we want examples in which LE sums and Lyapunov dimension are not maximized on equilibria. Therefore we choose one dissipative system and one Hamiltonian system that each have no equilibria.

For all computations reported below, Julia code to replicate the results is available at \href{http://github.com/jeremypparker/Parker_Goluskin_LEs}{\texttt{github.com/jeremypparker/Parker\char`_Goluskin\char`_LEs}}. We used SumOfSquares.jl \cite{sumofsquares} to reformulate SOS optimization problems as SDPs, along with SymbolicWedderburn.jl \cite{symbolicwedderburn} for symmetry exploitation. To solve the resulting SDPs, we primarily used Mosek version 11 \citep{mosek}. Certain SDPs were also solved using Clarabel.jl \citep{clarabel} and Hypatia.jl \citep{hypatia}, which gave very similar results with longer runtimes.

In SOS formulations from the sphere projection approach, the bounds $B$ can be minimized as the objective of the SDPs, and the bounds we report are the numerical optima returned by Mosek. In SOS formulations from the shifted spectrum approach, the value of $B$ must be fixed during each SDP solution because it multiplies coefficients of $V$ in the SOS constraints, and the coefficients to be optimized over can appear only linearly. In such computations, if there exists a $V$ satisfying the SOS constraints, the SDP solver reports `feasibile', and otherwise it reports `infeasible'. To minimize the $B$ at which bounds are verified in the shifted spectrum approach, we perform a bisection search in $B$, solving SDPs at various fixed $B$ to find the boundary between large $B$ where the SDPs are feasible and small $B$ where they are infeasible. As $B$ approaches this boundary, numerical conditioning of the SDPs worsens, making it hard to precisely identify the infimum over feasible $B$. Confirming that near-optimal bounds are valid calls for rigorous and/or multiple-precision computations \citep{goluskin2018bounding,parker2024lorenz}, but for simplicity we have used only double precision arithmetic in our examples. Guided by~\citep{parker2024lorenz}, we deem a bound feasible only if Mosek reports a dual objective value of less than $10^{-6}$.

In all cases, the maximum total degree of $V$ in the components of $x$ is restricted to some value $\nu_x$.
In the sphere projection formulations, we also specify the maximum total degree of $V$ in $z^k$ as some value $\nu_z$. Then the maximum overall degree of each term is $\max\{\nu_x,\nu_z\}$ when bounding $\Mu_k$ using $(x,z^k)$, or it is $\max\{\nu_x,2\nu_z\}$ when bounding $d_L$ using $(x,z^k,z^{k+1})$. In the shifted spectrum formulations, $V$ is quadratic in $w^k$ or $w$ when bounding $\Mu_k$ or $d_L$, respectively, and in the latter case the maximum total degree of $V$ in $y$ is $\nu_y$ and in $(x,y)$ is $\max\{\nu_x,\nu_y\}$. Spaces for polynomials other than $V$ are chosen based on the space for $V$. In particular, maximum degrees of $V$ dictate the maximum degrees of the expressions $S$ that are affine in $V$ and that are constrained to lie in a quadratic module~\cref{eq:quad module}. Maximum degrees of any polynomials $\sigma_i$ and/or $\rho_j$ appearing in the quadratic module representation are chosen such that this representation has the same maximum degrees as $S$. The exact polynomial spaces used in each of our computations for both examples can be seen in the accompanying code.

To confirm the sharpness of our numerical bounds, we also computed various periodic orbits and the LE sums along each of these orbits. We used the Tsit5 solver from the OrdinaryDiffEq.jl package \citep{DifferentialEquations.jl-2017} to integrate the ODE systems, find periodic orbits and compute LEs. For an orbit with period $T$, the Jacobian $D\phi^T$ of the one-period flow map was constructed by automatic differentiation using the ForwardDiff.jl package \citep{RevelsLubinPapamarkou2016}. Then, the real parts of the Jacobian's eigenvalues give the LEs $\mu_k$. From these $\mu_k$ we calculate the periodic orbit's LE sums $\Mu_k$ and Lyapunov dimension.

\subsection{The Duffing Oscillator}
\label{sec:duffing}
The Duffing oscillator is a second-order non-autonomous ODE governing the amplitude $A(t)$ of oscillations that are damped and sinusoidally forced. Different versions exist, but we choose
\begin{equation}
\label{eq:duffing}
\frac{{\rm d}^2A}{{\rm d}t^2} + \delta \frac{{\rm d}A}{{\rm d}t} + \beta A +\alpha A^3 = \gamma \cos{\omega t}
\end{equation}
with parameter values $(\alpha,\beta,\gamma,\delta,\omega)=(1,-1,0.3,0.2,1)$. Numerical integration gives a trajectory that apparently converges to a chaotic attractor. We numerically estimate the Lyapunov exponents on the chaotic attractor using ChaosTools.jl \citep{DynamicalSystems.jl-2018,DatserisParlitz2022}. This gives $\mu_1\approx0.158$ and $\mu_2\approx-0.358$, which imply a local Lyapunov dimension of roughly $1.441$. \citet{kuznetsov2020attractor} derived an analytical upper bound on the Lyapunov dimension $d_L$ of the global attractor, which gives  $d_L \leq 1.9910$ at our parameter values. (In their notation, the present values are $\delta=0.1$, $\varepsilon=0.3$ and $\omega=1$.) Note that we have not defined LEs or Lyapunov dimension for non-autonomous systems, although it is straightforward to do so. Instead, we can reformulate~\cref{eq:duffing} as a first-order system whose right-hand side is autonomous and polynomial, so that the SOS methods of \cref{sec:sos} are directly applicable.

To obtain a first-order system from the ODE~\cref{eq:duffing} we let $x_1 = A$ and $x_2 = \ddt{}{A}$. To obtain an autonomous polynomial system, we employ a standard trick (e.g., as in \citet{parker2021study}) and let $x_3=\sin{\omega t}$ and $x_4=\cos{\omega t}$. Then the $x(t)$ vector evolves according to
\begin{equation}
\label{eq:duffingsys}
\ddt{}{x} = f(x), \quad \text{where}\quad
f(x) =\begin{bmatrix}
        x_2\\-\delta x_2-\beta x_1 - \alpha x_1^3 + \gamma x_4\\
        \omega x_4\\
        -\omega x_3
    \end{bmatrix}.
\end{equation}
Since $x_3$ and $x_4$ are sine and cosine of the same argument, dynamics are restricted to the semialgebraic set
\begin{equation}
    \mathcal{B} = \{x\in\mathbb{R}^4~:~h(x)=0\}, 
    \quad \text{where}\quad h(x) = x_3^2+x_4^2-1.
\end{equation}
Only odd powers of $x$ appear in \cref{eq:duffingsys}, so the right-hand side is equivariant under the simple sign-change symmetry $\Lambda=-I$, and the domain $\mathcal{B}$ is $\Lambda$-invariant. This induces a corresponding symmetry in the augmented ODE system for each bounding formulation (cf.~\cref{thm:generalsymm}), along with the symmetry that is always present (cf.~\cref{thm:signsymm}). In our code, the symmetries of the relevant augmented system are imposed on $V$, and on any other polynomials appearing in Positivstellans\"atze, as explained in \cref{sec:syms}.

For the right-hand side of the ODE~\cref{eq:duffingsys}, the Jacobian matrix is 
\begin{equation}
    Df(x) = \begin{bmatrix}
        0 & 1 & 0 & 0\\
        -\beta -3\alpha x_1^2 & -\delta & 0 & \gamma\\
        0 & 0 & 0 &\omega \\
        0 & 0 & -\omega & 0
    \end{bmatrix}.
\end{equation}
The additive compound matrices of this Jacobian are
\begin{align}
\label{eq:duffing additivecompounds}
    Df^{[2]}(x) &= \begin{bmatrix}
        -\delta & 0 &\gamma & 0 & 0 & 0\\
        0&0&\omega&1&0&0 \\
        0&-\omega&0&0&1&0 \\
        0&-\beta-3\alpha x_1^2&0&-\delta&\omega&\gamma \\
        0&0&-\beta-3\alpha x_1^2&\omega&-\delta&0 \\
        0&0&0&0&0&0
    \end{bmatrix}, \\
    Df^{[3]}(x) &= \begin{bmatrix}
        -\delta & \omega & -\gamma & 0\\
        -\omega & -\delta & 0 & 0\\
        0 & 0 & 0 & 1 \\
        0 & 0 & -\beta-3\alpha x_1^2 & -\delta
    \end{bmatrix}, \\
  Df^{[4]}(x) &= -\delta,
\end{align}
and we recall that $Df^{[1]}=Df$ always. In this system, various relations for the LEs can be deduced \emph{a priori}. The fourth additive compound---i.e., the trace of the Jacobian---is pointwise constant here, which implies that $\Mu_4=-\delta$ on every trajectory. Furthermore, every trajectory has two zero LEs: one associated with constant rotation between the sine and cosine variables ($x_3$ and $x_4$), and another associated with translation along the flow, which applies to all trajectories since there are no equilibria. We expect all trajectories to have a single positive LE $\mu_1>0$, so that the four LEs are $(\mu_1,0,0,-(\delta+\mu_1))$. Note that the two zero LEs are created by the reformulation of the second-order ODE~\cref{eq:duffing} as the fourth-order system~\cref{eq:duffingsys}. Directly applying the definition \cref{eq:sumdef} using the time-dependent dynamics in \cref{eq:y flow map}, the original system would have only the two LEs $\mu_1$ and $-(\delta+\mu_1)$.

The LE relations described above for the system~\cref{eq:duffingsys} imply that, on each trajectory, the LE sums are related by $\Mu_1=\Mu_2=\Mu_3$. This further implies that the local Lyapunov dimension is $3+\Mu_1/(\delta+\Mu_1)$, which corresponds to $1+\Mu_1/(\delta+\Mu_1)$ for the original second-order equation~\cref{eq:duffing}. Therefore, to obtain sharp upper bounds on all maximum LE sums and on the global Lyapunov dimension, it suffices to find a sharp bound on the maximum leading LE $\Mu_1$. Nevertheless, in order to test our methods, we computed bounds using our formulations for $\Mu_1$, $\Mu_2$, $\Mu_3$ and $d_L$. We do not report computations for $\Mu_4$ since $V=0$ already gives the sharp bound $\Mu_4\le -\delta$, due to  $Df^{[4]}=-\delta$ holding pointwise in this system.

To formulate SOS optimization problems, we choose finite-dimensional spaces over which to optimize $V$ and the other polynomials appearing in Positivstellans\"atze. When using \cref{thm:projectionmethod sos} to bound $\Mu_k$ by the sphere projection approach, for instance, these other polynomials include $\rho$ that multiplies $h$ and $\sigma$ that multiplies $|z^k|^2-1.$ The spaces of these polynomials are restricted by imposing symmetries of the problem as described in \cref{sec:syms}. They are also restricted by imposing maximum polynomial degrees of $V$ as described in the fourth paragraph of \cref{sec:examples}.

\Cref{tab:duffing} reports bounds computed using the sphere projection approach. Bounds on the LE sums $\Mu_k$ and Lyapunov dimension $d_L$ are computed by solving the SOS optimization problems \cref{eq:C1 def,eq:C2 def}, respectively, for various polynomial degrees $\nu_x$ and $\nu_z$.
\begin{table}[t]
\centering
\caption{Sphere projection approach for the Duffing oscillator in its autonomous formulation~\cref{eq:duffingsys}. The leading LE sums $\Mu_k$ are bounded by solving the SOS problem in \cref{thm:projectionmethod sos}, excluding $\Mu_4$ since it is trivial in this system (see text). The Lyapunov dimension $d_L$ of the global attractor is bounded by solving the SOS problem in \cref{thm:dimcombined sos} with $k=3$; bounds for the non-autonomous formulation~\cref{eq:duffing} can be obtained by subtracting~2. The auxiliary function $V$ is restricted to various maximum degrees $(\nu_x,\nu_z)$ in $x$ and $z^k$, which induce corresponding degrees in other polynomials, as described in the text. The tabulated bounds are rounded to the precision shown, and the underlines indicate digits that agree (after rounding) with the numerically computed values $\Mu_1=\Mu_2=\Mu_3\approx0.887897$ and $d_L\approx3.816159$ attained by the periodic orbit shown in \cref{fig:duffing}.}
\label{tab:duffing}
\vspace{6pt}
\begin{tabular}{c c c c c c c}
\hline
\multirow{2}{*}{$\nu_x$}&\multirow{2}{*}{$\nu_z$} & \multicolumn{4}{c}{Upper bounds} \\ \cline{3-6}
&& $\Mu_1$ & $\Mu_2$ & $\Mu_3$ & $d_L$ \\ \hline 
2&0&27.36247& 31.82308&29.78229 & \underline{3.}99116 \\ 
&2&14.82952&15.58938&13.93659 & \underline{3.}97731 \\
&4&7.14087&7.86722&7.73049 &\underline{3.}96131 \\ \hline
4&0&\underline{1.}36659&\underline{1.}36676&\underline{1.}36662& \underline{3.8}7232\\
&2&\underline{0.88790}&\underline{0.8879}7&\underline{0.88790} & \underline{3.81616} \\
&4&\underline{0.88790}& \underline{0.88790}&\underline{0.88790} & \underline{3.81616} \\ \hline
6 & 0 &\underline{1.}18428&\underline{1.}18427&\underline{1.}18481&\underline{3.8}5544 \\
& 2 &\underline{0.88790}&\underline{0.8879}7&\underline{0.88790}& \underline{3.81616} \\ \hline
\end{tabular}
\end{table}
The number of variables in the SOS program~\cref{eq:C1 def} for $\Mu_k$ with $k=1,2,3$ are $4+\bn{4}{k}=8,10,8$, respectively, and the number of variables in the SOS program~\cref{eq:C2 def} for $d_L$ is $4+\bn{4}{3}+\bn{4}{4}=9$. All bounds converge quickly as polynomial degrees are raised, and our best bounds on each quantity are sharp to at least 4 digits, as we confirm below by finding a periodic orbit whose LE sums and Lyapunov dimension saturate the bounds to this precision. We will see below that the maximizing orbit is geometrically simple, which typically allows for faster convergence as polynomial degrees are raised, relative to examples with more complicated maximizing orbits \citep{goluskin2018bounding}.

We have also computed bounds using the shifted spectrum approach. In this case, the SOS optimization problems solved to bound $\Mu_k$ and $d_L$ are \cref{thm:shiftedmethod sos,thm:shifted dimension sos}, respectively. All expressions are quadratic in the $w$ variables, and we must choose only the total degree in $x$ for $V$ and for other optimized polynomials. Feasibility or infeasibility of the SOS constraints for the chosen polynomial spaces is determined by solving SDPs with various fixed $B$, and the minimum feasible $B$ is sought by a bisection search. With $\nu_x\le 4$, numerical errors are large, and we did not find clear feasibility at any $B$ value. With $\nu_x=6$, however, we obtain the bounds
\begin{equation}
\Mu_1\le \underline{0.887898} \quad \text{and} \quad d_L\le \underline{3.8161}7,
\label{eq:duffing shifted bounds}
\end{equation}
where underlines indicate digits whose sharpness is confirmed by the periodic orbit described below. The SOS problem \cref{eq:C3 def,eq:C4 def}, which are feasible when $B$ is fixed to the respective bound values in~\cref{eq:duffing shifted bounds} with $\nu_x=6$, do not meet our criterion for feasibility when the final digits of the bounds in~\cref{eq:duffing shifted bounds} are decreased by $1$.

\begin{figure}[t]
    \centering
    \includegraphics[width=0.7\linewidth]{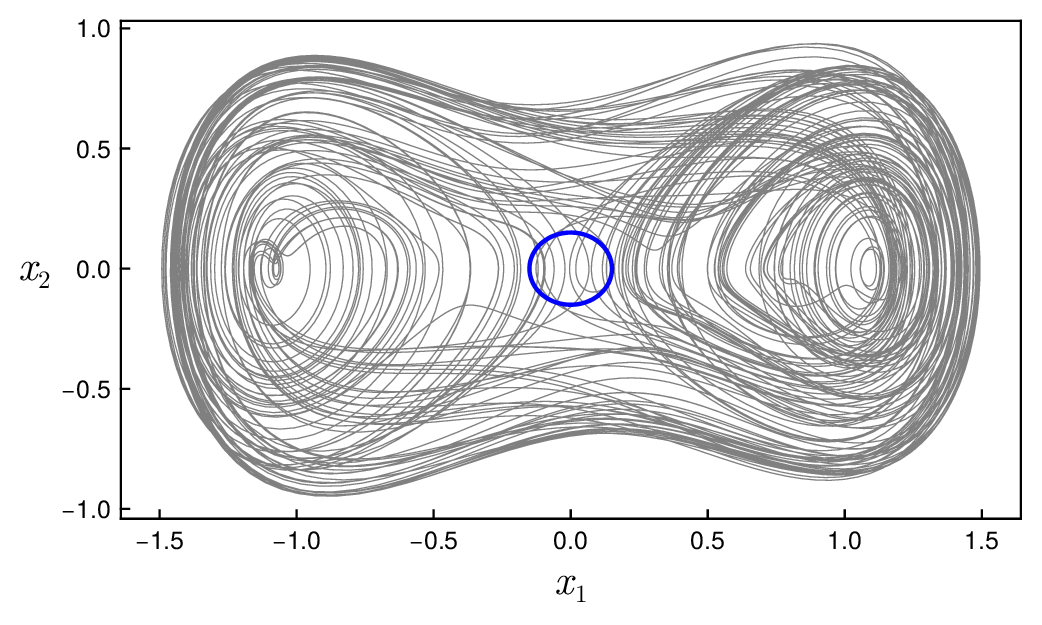}
    \caption{A trajectory on the chaotic attractor (thin line) and the unstable $2\pi$-periodic orbit (thick line) for the Duffing oscillator. The figure is the same for the non-autonomous system~\cref{eq:duffing} or its autonomous reformulation \cref{eq:duffingsys}.}
    \label{fig:duffing}
\end{figure}

To seek an orbit that maximizes the leading LE $\Mu_1$ among all trajectories, and consequently also maximizes the other $\Mu_k$ and the Lyapunov dimension, we follow ideas in~\citep{tobasco2018optimal,lakshmi2020finding} for using the $V$ found in SOS computations for bounding time averages. Applied here, their observations imply that that, for any $V(x,z^1)$ giving a nearly sharp bound $B$ on $\Mu_1$ via the SOS program~\cref{eq:C1 def}, the time average of
\begin{equation}
B-\Phi_1 - f\cdot D_xV-\ell_1\cdot D_{z^1}V
\label{eq:near-optimal set}
\end{equation}
is small on any orbit that nearly maximizes $\Mu_1$. This quantity is also constrained to be nonnegative, so it must be small pointwise on most parts of such orbits. Therefore we seek points in the $(x,z^1)$ state space where \cref{eq:near-optimal set} is small and then seek periodic orbits passing near these points.

To find points where~\cref{eq:near-optimal set} is small, we fix $(x_3,x_4)=(0,1)$ and consider a mesh of $(x_1,x_2)$ over roughly the extent of the chaotic attractor. At each $(x_1,x_2)$ pair, we approximate the minimum of~\cref{eq:near-optimal set} over $z^1\in\mathbb{S}^3$ by random sampling. For the $(x,z)$ at which we find our smallest value of~\cref{eq:near-optimal set}, this $x$ is used as the initial guess for a Newton-shooting algorithm, which converges to a $2\pi$-periodic orbit that approximately intersections the point
\begin{equation}
    (x_1,x_2,x_3,x_4) = (-0.14984, 0.01520, 0, 1).
\end{equation}
This is the small, approximately elliptical periodic orbit in the center of \cref{fig:duffing}. We computed LEs on this orbit to be approximately
\begin{equation}
    \mu_1=0.887897,\quad\mu_2=0,\quad\mu_3=0,\quad\mu_4= -1.08785.
\end{equation}
These LEs give
\begin{equation}
    \Mu_1=\Mu_2=\Mu_3=0.887897,\quad\Mu_4= -0.20000, \quad d_L=3.816159,
    \label{eq:duffing PO Mk}
\end{equation}
where $d_L$ is computed from the Kaplan--Yorke formula~\cref{eq:dL} with the assumption that $\Mu_1$ is indeed maximized on this periodic orbit. The values in~\cref{eq:duffing PO Mk} match at least 4 digits of our best upper bounds from both the sphere projection approach in \cref{tab:duffing} and the shifted spectrum approach in~\cref{eq:duffing shifted bounds}. This confirms sharpness of these bounds and maximality of the period orbit in \cref{fig:duffing} to at least 4 digits. In the limit $\gamma\to0$ of \cref{eq:duffing}, the $2\pi$-periodic orbit becomes an equilibrium at the origin. This is consistent with previous observations that Lyapunov dimension is maximized on equilibria in systems that have such points. 

For the Duffing oscillator in its original non-autonomous formulation~\cref{eq:duffing}, the sharp bound $d_L\le 1.8162$ on global Lyapunov dimension follows from the bound $d_L\le3.8162$ that we computed for the autonomous reformulation~\cref{eq:duffingsys}. This improves on the analytical bound from \cite{kuznetsov2020attractor}, which gives $d_L \leq 1.9910$ at the present parameters.

\subsection{A three-degree-of-freedom Hamiltonian system\label{sec:hamiltonian}}

For an example in which LE sums $\Mu_k$ for different $k$ are not all maximized on the same orbits, we consider a 6-dimensional ODE whose numerous symmetries keep the computational costs tractable. In particular, we consider a Hamiltonian system whose Hamiltonian function is
\begin{equation}
    \label{eq:triaxial}
    H(x) = \frac{1}{2}\left|x\right|^2 + x_1^2x_2^2+x_2^2x_3^2+x_3^2x_1^2,
\end{equation}
where the first three components of $x\in\mathbb{R}^6$ are position variables, and the other three components are corresponding momenta. This is a specific form of a triaxially symmetric potential studied by \citet{palacian2017periodic} that has applications in galactic dynamics \citep{de1985motion,caranicolas19941}. The Hamiltonian~\cref{eq:triaxial} gives rise to the ODE system
\begin{equation}
\ddt{}x = f(x), \quad \text{where} \quad
f(x)= \begin{bmatrix}
        x_4\\x_5\\x_6\\
        -x_1-2x_1x_2^2-2x_1x_3^2\\
        -x_2-2x_2x_1^2-2x_2x_3^2\\
        -x_3-2x_3x_1^2-2x_3x_2^2
    \end{bmatrix}.
    \label{eq:triaxialode}
\end{equation}
We consider dynamics only on the $H=1$ energy surface, where there are no equilibria. Our semialgebraic specification of this domain also includes the redundant constraint that $|x|^2\le 2$, so it is
\begin{equation}
\mathcal{B} = \bigg\{x\in\mathbb{R}^3\;:\;g(x)\ge0,~h(x)=0\bigg\},
\label{eq:hamiltonian domain}
\end{equation}
where 
\begin{equation}
g(x)=2-|x|^2,\quad 
h(x)=H(x)-1.
\end{equation}
Adding the inequality constraint does not change the set $\mathcal{B}$, but it changes the SOS Positivestelans\"atze that certify nonnegativity on $\mathcal{B}$, and we find empirically that this gives better bounds at lower computational cost in this example. The ODE~\cref{eq:triaxialode} and its domain~\cref{eq:hamiltonian domain} are both invariant under $x\mapsto\Lambda x$ for all $\Lambda$  in the 48-element octahedral symmetry group. Represented as a subgroup of $O(6)$, this group is generated by a sign-change $\Lambda_1$ and permutations $\Lambda_2$ and $\Lambda_3$, where
\begin{equation*}
    \Lambda_1 = \begin{bmatrix}
        -1 & 0 & 0 & 0 & 0 & 0\\
        0 & 1 & 0 & 0 & 0 & 0\\
        0 & 0 & 1 & 0 & 0 & 0\\
        0 & 0 & 0 & -1 & 0 & 0\\
       0& 0 & 0 & 0 & 1 & 0\\
        0& 0 & 0 & 0 & 0 & 1\\
    \end{bmatrix},\ 
    \Lambda_2 = \begin{bmatrix}
        0 & 1 & 0 & 0 & 0 & 0\\
        1 & 0 & 0 & 0 & 0 & 0\\
        0 & 0 & 1 & 0 & 0 & 0\\
        0 & 0 & 0 & 0 & 1 & 0\\
       0& 0 & 0 & 1 &0 & 0\\
        0& 0 & 0 & 0 & 0 & 1\\
    \end{bmatrix},\ 
    \Lambda_3 = \begin{bmatrix}
        0 & 0 &1 & 0 & 0 & 0\\
        0 & 1 & 0 & 0 & 0 & 0\\
        1 & 0 & 0 & 0 & 0 & 0\\
        0 & 0 & 0 & 0 & 0 & 1\\
       0& 0 & 0 & 0 & 1 & 0\\
        0& 0 & 0 & 1 & 0 & 0\\
    \end{bmatrix}.
\end{equation*}

Various relations between the LEs hold \emph{a priori} because this system is Hamiltonian and has no equilibria. Hamiltonian systems conserve state space volume, meaning that $\mathrm{tr}(Df)=0$ pointwise. This implies that  $\Mu_6=0$ on all trajectories, and that there are no attractors, so our methods for bounding Lyapunov dimension are not relevant. Hamiltonian structure also implies that LEs come in pairs of $\mu_k$ with opposite sign \citep{skokos2009lyapunov}, so $(\mu_1,\mu_2,\mu_3)=(-\mu_6,-\mu_5,-\mu_4)$. Since there are no equilibria, all trajectories have a zero LE corresponding to time translation, which in Hamiltonian systems guarantees at least two zero LEs, here meaning $\mu_3=\mu_4=0$.

Relations between the LE sums $\Mu_k$ follow from the $\mu_k$ relations, which imply
\begin{equation}
\Mu_1=\Mu_5=\mu_1, \quad
\Mu_2=\Mu_3=\Mu_4=\mu_1+\mu_2, \quad 
\Mu_6 = 0.
\end{equation}
Therefore, to obtain sharp upper bounds on the maxima of all LE sums, it suffices to compute sharp bounds on $\Mu_1$ and $\Mu_2$. Nevertheless, in order to test our methods, we computed bounds for all $\Mu_k$ aside from $\Mu_6$. The fact that $\Mu_6= 0$ on all trajectories is immediate since $Df^{[6]}=0$ pointwise, and simply taking $V=0$ in our bounding formulations would give $\Mu_6\le 0$.

We have computed a number of unstable periodic orbits lying in $\mathcal{B}$. The maxima of $\Mu_1$ and $\Mu_2$ among these orbits provide candidates for the maxima over all trajectories in $\mathcal{B}$, but this can only be confirmed by the upper bounds on $\Mu_k$ reported below. The periodic orbits were found by a brute-force search over random initial conditions, uniformly distributed over the unit box in $(x_1,x_2,x_4,x_5)$ on the $x_3=0$ Poincar\'e section, with $x_6$ chosen so that $H(x)=1$. These were converged by a Newton-like method to periodic orbits that close after one iteration of the Poincar\'e return map.  Among the periodic orbits we found, $\Mu_1$ is maximized on each orbit in a set of 6 symmetry-related orbits that we call $S_1$, while $\Mu_2$ is maximized on each in a set of 4 symmetry-related orbits that we call $S_2$. \Cref{fig:triaxial_POs} shows the $S_1$ and $S_2$ orbits plotted in position space.
\begin{figure}[t]
    \centering
    \includegraphics[width=0.49\linewidth]{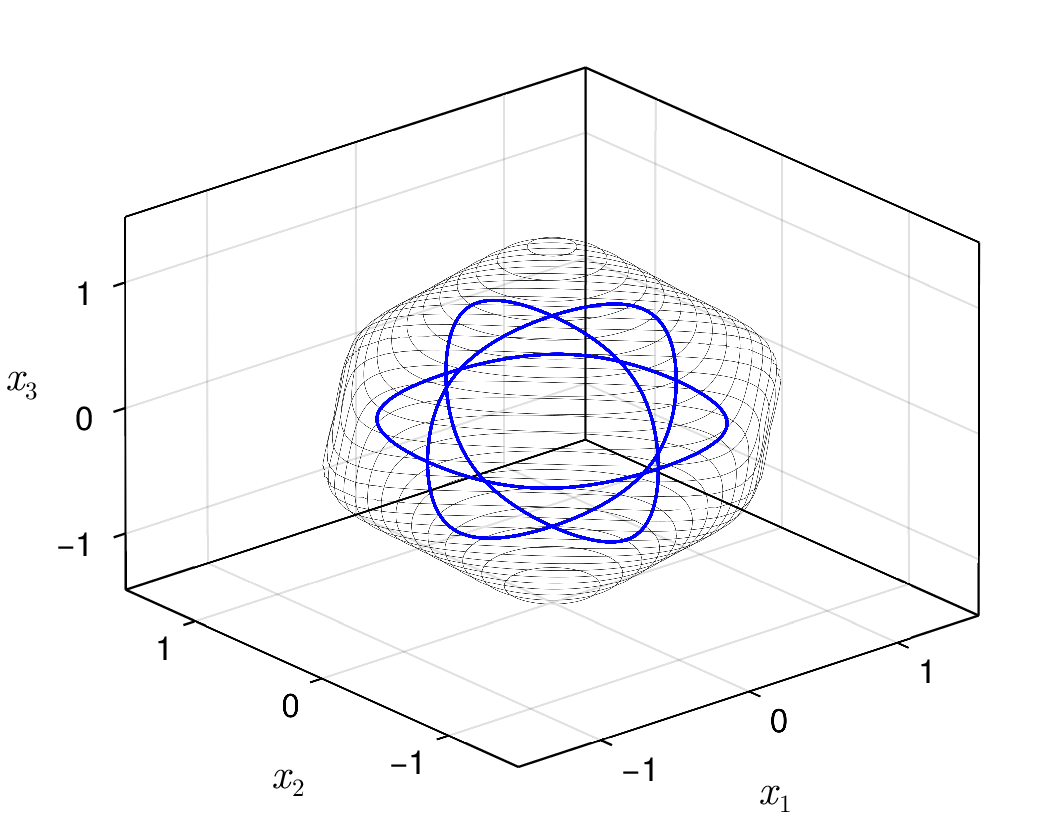}
    \includegraphics[width=0.49\linewidth]{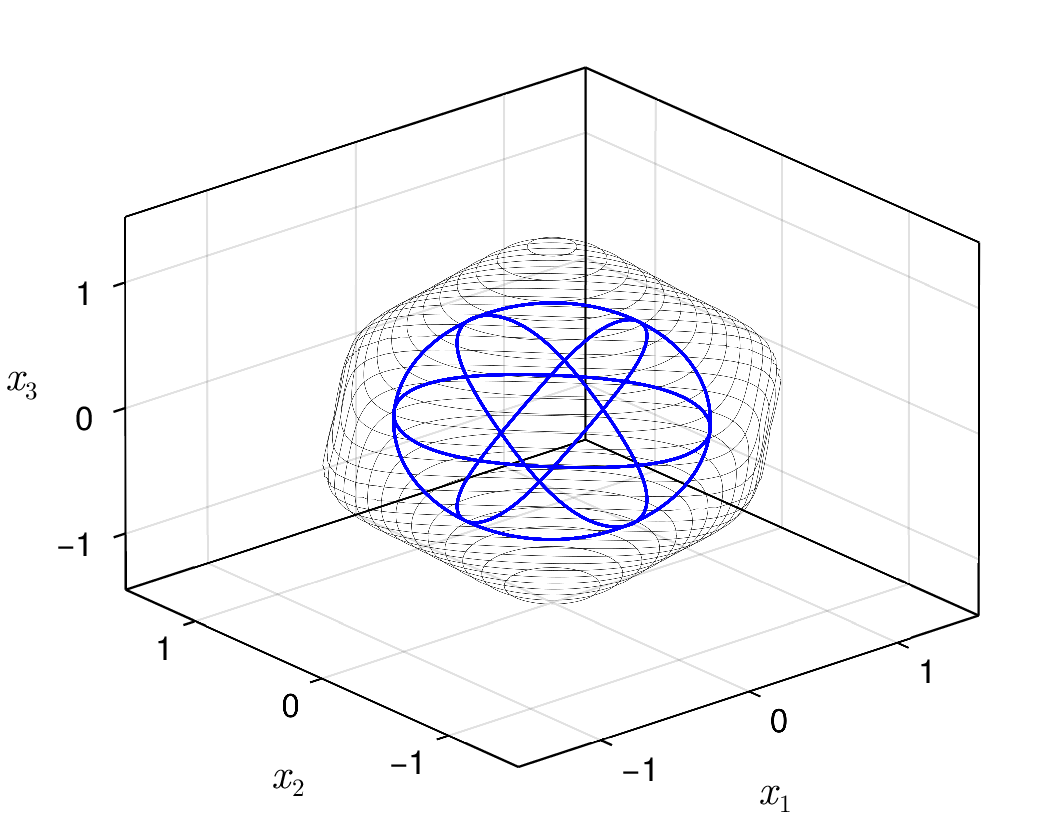}
    \caption{Position coordinates $(x_1,x_2,x_3)$ for the Hamiltonian system \cref{eq:triaxial} showing the $S_1$ (left) and $S_2$ (right) families of periodic orbits (thick lines). The $S_1$ orbits have LE sums $\Mu_1=\Mu_2=\Mu_3=\Mu_4=\Mu_5=0.42913$, and the $S_2$ orbits have $\Mu_1=\Mu_5=0.25919$ and $\Mu_2=\Mu_3=\Mu_4=0.51837$. The thin lines depict the part of the $H=1$ surface with zero momenta, which is the boundary of possible position coordinates.}
    \label{fig:triaxial_POs}
\end{figure}
The $S_1$ orbits have a period of approximately 5.278 and pass near one of the symmetry-related points $\Lambda x$ for
\[x=(0.83871,0,0,0,0,1.13867).\]
Their LEs are approximately
\[(\mu_1,\mu_2,\mu_3,\mu_4,\mu_5,\mu_6)=(0.42913,0,0,0,0,-0.42913),\]
corresponding to LE sums of 
\[(\Mu_1,\Mu_2,\Mu_3,\Mu_4,\Mu_5,\Mu_6)=(0.42913,0.42913,0.42913,0.42913,0.42913,0).\]
The $S_2$ orbits have a period of approximately 4.857 and pass near $\Lambda x$ for
\[ x=(-0.57566,0.57566,0,0.39004,0.39004,0.90186). \]
Their LEs are approximately\[(\mu_1,\mu_2,\mu_3,\mu_4,\mu_5,\mu_6)=(0.25919,0.25919,0,0,-0.25919,-0.25919),\] 
corresponding to sums of 
\[(\Mu_1,\Mu_2,\Mu_3,\Mu_4,\Mu_5,\Mu_6)=(0.25919,0.51837,0.51837,0.51837,0.25919,0).\] 
One may then guess that the maximum of $\Mu_1=\Mu_5$ among all trajectories in $\mathcal{B}$ is close to 0.42913, and that the maximum of $\Mu_2=\Mu_3=\Mu_4$ is close to 0.57566. These guesses are confirmed by the upper bounds that we report next.

As in the example of \cref{sec:duffing}, we formulate SOS optimization problems by choosing finite-dimensional spaces over which to optimize $V$ and the other polynomials appearing in Positivstellans\"atze. These spaces enforce symmetries of the problem (cf.~\cref{sec:syms}) and maximum polynomial degrees (cf.~the fourth paragraph of \cref{sec:examples}). We computed bounds on $\Mu_k$ using the SOS programs~\cref{eq:C1 def,eq:C3 def} from the sphere projection and shifted spectrum approaches, respectively. In both cases the total number of variables in the SOS program for $k=1,2,3,4,5$ are $6+\bn{6}{k}=12,21,26,21,12$, respectively. Due to high computational cost with this number of variables, the degree of $V$ in the $\bn{6}{k}$ tangent space variables was no larger than 2. The SOS constraints involve the additive compounds of the Jacobian matrix of $f$ in~\cref{eq:triaxialode}. For brevity we do not write down these compounds, as we did in the example of~\cref{sec:duffing}. They can be found in our included code, where they are constructed using the Julia package DynamicPolynomials.jl~\citep{benoit_legat_2024_11106280}.

\Cref{tab:triaxial} reports bounds on $\Mu_k$ computed using our two approaches. In both approaches, we vary the degree $\nu_x$ of $V$ in $x$, and in the sphere projection approach we also vary its degree $\nu_z$ in $z^k$.
\begin{table}[t]
\centering
\caption{Upper bounds on the LE sums $\Mu_k$ for the Hamiltonian system \cref{eq:triaxialode}, computed with the sphere projection and shifted spectrum approaches by solving the SOS problems in \cref{thm:projectionmethod sos,thm:shiftedmethod sos}, respectively. Polynomial degrees in $x$ are at most $\nu_x$. In the sphere projection approach, degrees in $z^k$ are at most $\nu_z$, and tabulated values are rounded up to the precision shown. In the shifted spectrum, expressions are quadratic in $w$, the tabulated values are the exact $B$ whose feasibility was verified, and feasibility could not be verified when the final digit of $B$ was decreased by 1. Underlines indicate digits that agree (after rounding) with the numerically computed values $\Mu_1=\Mu_5\approx0.42913$ on the $S_1$ periodic orbits or $\Mu_2=\Mu_3=\Mu_4\approx0.51837$ on the $S_2$ orbits (cf.~\cref{fig:triaxial_POs}). Bounds are missing in cases where solution of SDPs by interior-point algorithms required more than 3TB of memory.}
\label{tab:triaxial}
\vspace{6pt}
\begin{tabular}{c c c c c c c}
\hline
$\nu_x$ & $\nu_z$ & $\Mu_1$ & $\Mu_2$ & $\Mu_3$ & $\Mu_4$ & $\Mu_5$ \\\hline
&& \multicolumn{5}{c}{Upper bounds (sphere projection approach)} \\ 
\cline{3-7}
2 & 0 & 1.2361 & 2.0232 & 2.1082 & 2.0232 & 1.2361 \\
 & 2 & \underline{0.4}843 & \underline{0.}9216 & \underline{0.}9480 & \underline{0.}9216 & {0.4}843 \\ 
4 & 0 &  1.1855 & 2.0046 & 2.0611 & 2.0046 & 1.1855 \\
 & 2 & \underline{0.4292} & & & & \underline{0.4292} \\ \hline
&& \multicolumn{5}{c}{Upper bounds (shifted spectrum approach)} \\
\cline{3-7}
2 && \underline{0.43}11 & \underline{0.}6141  & \underline{0.}7579  & \underline{0.}6141 & \underline{0.43}11 \\
4 && \underline{0.4292} & \underline{0.5184} & \underline{0.5}711 & \underline{0.5184} & \underline{0.4292} \\
6 && \underline{0.4292} & \underline{0.5184} & \underline{0.5}300 & \underline{0.5184} & \underline{0.4292} \\         
\end{tabular}
\end{table}
Using the sphere projection, bounds on $\Mu_1$ and $\Mu_5$ are sharp to 4 digits with $V$ of maximum degrees $(\nu_x,\nu_z)=(4,2)$. For $k=2,3,4$, where the SOS problems have more variables, we did not obtain sharp bounds using sphere projection. With degrees $(\nu_x,\nu_z)=(4,2)$ at these $k$, we could not solve the necessary SDPs using Mosek because the 3TB of memory available to us was insufficient. This could be surmounted by using a first-order algorithm, rather than an interior-point algorithm, at the expense of very long runtimes and/or lower precision. Using the shifted spectrum approach, however, we obtain bounds that are sharp to 4 digits on all $\Mu_k$ aside from $k=3$. Bounds in the $k=3$ case are sharp to only 2 digits. This might be due to numerical inaccuracy in the solution of the SDP, which is larger than in any other case, or it might indicate a need for computations with $\nu_x=8$. Nonetheless, we remind the reader that we computed bounds with $k=3,4,5$ only to test our computational formulations. Due to the Hamiltonian structure of this example, sharp bounds on $\Mu_1$ and $\Mu_2$ immediately imply sharp bounds on the other $\Mu_k$ also. Therefore, we have indeed obtained bounds on all $\Mu_k$ that are sharp to 4 digits.

\section{Conclusion}
\label{sec:conclusion}

We have described two distinct approaches for computing upper bounds on sums of Lyapunov exponents and on Lyapunov dimension in ODE systems. One we call the sphere projection approach, which leads to a deterministic version of the Furstenberg--Khasminskii formula from stochastic dynamics. The other we call the shifted spectrum approach. In both approaches, our most general formulations require optimizing over auxiliary functions in the class of $C^1$ functions. We have proved that most of these formulations give sharp bounds under mild assumptions, but the optimization problems are not tractable in general. In the case of ODEs with polynomial right-hand sides, our bounding formulations can be relaxed using SOS constraints, yielding SOS optimization problems that are often computationally tractable, and which often converge to sharp bounds as polynomial degrees are raised. By carrying out such SOS computations for particular ODEs, we obtained nearly perfect bounds in challenging examples. For the Duffing oscillator example, the sphere projection computations had much better numerical conditioning. For the Hamiltonian example, the shifted spectrum computations were easier because the sphere projection computations had onerous memory requirements. In both examples, we confirmed sharpness of our best bounds by finding periodic orbits that saturate the bounds to at least 4 digits. We also demonstrated, in the Duffing oscillator example, how the results of our bounding computations can be used to find previously undetected orbits that saturate the bounds.

All of our computational methods could be used to construct fully rigorous bounds via computer-assisted proofs, making use of interval arithmetic or rational projection \citep{goluskin2018bounding, parker2024lorenz, parker2025computation}. In principle, computer-assisted proofs based on our framework could also give rigorous upper bounds on LEs in partial differential equations, by combining our methods with some in \citet{goluskin2019bounds}, but such computations would be challenging.

A clear future direction is to adapt our present methods to discrete-time dynamical systems. A formulation analogous to the sphere projection approach was formulated for bounding the single leading LE of discrete-time systems in the thesis of \citet{oeri2023thesis}, including computational formulations with SOS optimization in the case of polynomial maps. This can be directly extended to sums of LEs by using multiplicative compound matrices, as opposed to the additive compounds we have used for ODE dynamics. Formulations analogous to our shifted spectrum approach are also possible for discrete-time systems. For both approaches, the SOS computations arising for polynomial maps are generally harder than the SOS computations for polynomial ODEs that we have carried out here. The reason is that the auxiliary function $V$ is composed with the map in the discrete-time case, leading to high polynomial degrees. Explanation of this and other differences in the discrete-time case is left for a future publication.

\section*{Data availability statement}
Our code available at \href{http://github.com/jeremypparker/Parker_Goluskin_LEs}{\texttt{github.com/jeremypparker/Parker\char`_Goluskin\char`_LEs}} was used to carry out all computational examples.

\section*{Acknowledgements}
The authors thank Samuel Punshon-Smith and Anthony Quas for useful discussions, and James Meiss for suggesting possible Hamiltonian examples. During this work DG was supported by the NSERC Discovery Grants Program (awards RGPIN-2018-04263 and RGPIN-2025-06823). Computational resources were provided by the Digital Research Alliance of Canada.

\appendix
\crefalias{section}{appsec}
\crefalias{subsection}{appsec}
\crefalias{subsubsection}{appsec}

\section{Exterior products and compound matrices}
\label{sec:compounds}
The exterior algebra on $\mathbb{R}^n$, also called the Grassman algebra, is a graded algebra with a product that is denoted by $\wedge$ and called the exterior product. An exterior product of $k$ vectors in $\mathbb{R}^n$ is called a $k$-blade, and the space spanned by these is the space of $k$-vectors. The properties of the exterior product imply that permuting the $x_i$ in $x_1\wedge\dots\wedge x_k$ either leaves the $k$-blade unchanged or negates it, and that the $k$-blade is nonzero if and only if the $x_i$ are linearly independent. The set of $k$-vectors over $\mathbb{R}^n$ is therefore a linear space of dimension $\bn{n}{k}=\tfrac{n!}{k!(n-k)!}$. We identify this space with $\mathbb{R}^{\bn{n}{k}}$ by choosing a basis, which is the set of $k$-blades of the form $e_{i_1}\wedge\dots\wedge e_{i_k}$, where $e_i$ denotes the $i^\mathrm{th}$ standard basis vector in $\mathbb{R}^n$, and the indices $i_1<\ldots<i_k$ are all possible tuples of $k$ increasing integers between 1 and $n$. The standard Euclidean norm on this space satisfies the property that $|x_1\wedge\dots\wedge x_k|$ is the volume of the parallelepiped spanned by $x_1,\dots,x_k\in\mathbb{R}^n$ \citep{winitzki2009linear}.

Any linear map $A$ on vectors in $\mathbb{R}^n$ induces a linear map on corresponding $k$-vectors in $\mathbb{R}^{\bn{n}{k}}$, called the multiplicative compound $A^{(k)}$. Likewise, any linear ODE on $\mathbb{R}^n$ with right-hand side $A$ induces a linear ODE on $k$-vectors whose right-hand side is called the additive compound $A^{[k]}$. The multiplicative compound arises in many areas of mathematics and is known variously as the \textit{compound representation}, \textit{exterior power representation}, \textit{wedge power representation} or simply the \textit{induced action} of a linear map. Compound matrices have been applied previously to dynamical systems in the calculation and bounding of LEs \citep{manika2013application,martini2022ruling,martini2023bounding,kuznetsov2020attractor} and in stability analysis \citep{muldowney1990compound}. Here we state key definitions and theorems; a more detailed introduction can be found in \citet[Chapter 19F]{marshall2010inequalities}, in more generality for complex non-square matrices.

\begin{defn}
\label{def:comp}
For any real matrix $A\in \mathbb{R}^{n\times n}$ and integer $k\in\{1,\ldots,n\}$:
\begin{enumerate}
    \item The $k^\mathrm{th}$ \textit{multiplicative compound} of $A$ is the matrix $A^{(k)}\in\mathbb{R}^{{\bn{n}{k}} \times \bn{n}{k}}$ whose entries are the $k\times k$ subdeterminants of $A$, taken in lexicographic order. 
    \item The $k^\mathrm{th}$ \textit{additive compound} of $A$ is the matrix $A^{[k]}\in\mathbb{R}^{{\bn{n}{k}} \times \bn{n}{k}}$ such that
    \begin{equation}
        A^{[k]} = \frac{\mathrm{d}}{\mathrm{d}\delta}\left.\left(I+\delta A\right)^{(k)}\right|_{\delta=0}.
    \end{equation}
\end{enumerate}
\end{defn}
The lexicographic order in \cref{def:comp} refers to the ordering of all $\bn{n}{k}$ permutations of $k$ increasing integers chosen from $\{1,\ldots,n\}$. In this ordering, tuples are ordered by their first element, those with the same first element are ordered by their second element, and so on. As an example, the lexicographic order of $k=3$ integers in the case $n=4$ is
\begin{equation}
(1,2,3),\quad (1,2,4),\quad (1,3,4), \quad (2,3,4).
\end{equation}
Therefore, the $(2,3)$ entry of $A^{(3)}$ is
\begin{equation}
A^{(3)}_{2,3}=\det \begin{bmatrix} a_{1,1} & a_{1,3} & a_{1,4}\\ a_{2,1} & a_{2,3} & a_{2,4} \\ a_{4,1} & a_{4,3} & a_{4,4}
\end{bmatrix},
\end{equation}
Note that the row and column indices come from the second and third permutations in the lexicographic ordering, respectively. It follows from the definitions that $A^{(1)}=A^{[1]}=A$, while $A^{(n)}=\det{A}$ and $A^{[n]}=\mathrm{tr} A$. The eigenvalues of the $k^\mathrm{th}$ multiplicative and additive compounds are the products and sums, respectively, of the $k$ eigenvalues of $A$, which is why the compounds are called `multiplicative' and `additive'. Additive compounds for all $k$ in an explicit example with $n=4$ are given in \cref{eq:duffing additivecompounds}.

The next two propositions give properties of multiplicative and additive compounds, respectively, that will be useful to us. In \cref{thm:multcompoundprops}, part 1 is a special case of the Cauchy--Binet theorem, parts 2--4 follow quickly from the definition, and part 5 is theorem 6.2.10 of \citet{fiedler2008special}. For $A^{(k)}$ and $A^{[k]}$ to be linear mappings on $k$-blades, as in part 5 of \cref{thm:multcompoundprops} and part 2 of \cref{thm:addcompoundprops}, we identify the space of $k$-vectors with $\mathbb{R}^{\bn{n}{k}}$.
\begin{prop}[Properties of multiplicative compound matrices]
\label{thm:multcompoundprops}
For any $A,B\in\mathbb{R}^{n\times n}$:
\begin{enumerate}
\item $(AB)^{(k)}=A^{(k)}B^{(k)}$.
\item $\left(A^{(k)}\right)^\mathsf{T} = \left(A^\mathsf{T}\right)^{(k)}$.
\item    If $A$ is invertible, then $\left(A^{-1}\right)^{(k)}=\left(A^{(k)}\right)^{-1}$.
\item $A$ is orthogonal if and only if $A^{(k)}$ is orthogonal.
\item For any $k$-blade $x_1\wedge\dots\wedge x_k$,
\begin{equation}
    A^{(k)} \left(x_1\wedge\dots\wedge x_k\right) = (Ax_1)\wedge\dots\wedge (A x_k).
    \end{equation}
\end{enumerate}
\end{prop}

\begin{prop}[Properties of additive compound matrices]
\label{thm:addcompoundprops}
For any $A,B\in\mathbb{R}^{n\times n}$:
\begin{enumerate}
\item  $\left(A+B\right)^{[k]} = A^{[k]} + B^{[k]}$.
\item For any $k$-blade $x_1\wedge\dots\wedge x_k$,
\begin{equation}
A^{[k]} \left(x_1\wedge\dots\wedge x_k\right) = (A x_1\wedge x_2\wedge\dots\wedge x_k) + (x_1\wedge A x_2\wedge\dots\wedge x_k) + \dots + (x_1\wedge x_2\wedge\dots\wedge A x_k).
\end{equation}
\item If $B$ is invertible, $\left(B A B^{-1}\right)^{[k]} = B^{(k)}A^{[k]}\left(B^{(k)}\right)^{-1}$. 
\end{enumerate}
\end{prop}
\begin{proof}
Part 1 is immediate from the definition. For part 2, see theorem 2f of \citep{london1976derivations}. Part 3 follows by calculation,
    \begin{align*}
        \left(B A B^{-1}\right)^{[k]} &= \frac{\mathrm{d}}{\mathrm{d}\delta}\left.\left(I+\delta BAB^{-1}\right)^{(k)}\right|_{\delta=0}\\
        &= \frac{\mathrm{d}}{\mathrm{d}\delta}\left.\left[B\left(I+\delta A\right)B^{-1}\right]^{(k)}\right|_{\delta=0}\\
        &= B^{(k)} \frac{\mathrm{d}}{\mathrm{d}\delta}\left.\left(I+\delta A\right)^{(k)}\right|_{\delta=0}\big(B^{-1}\big)^{(k)} \\
        &= B^{(k)}A^{[k]}\big(B^{(k)}\big)^{-1},
    \end{align*}
    where the third equality uses property 1 from  \cref{thm:multcompoundprops}, and the fourth equality uses property 3.
\end{proof}
Multiplicative compound matrices are straightforward to compute directly, at least when $n$ is small. Additive compound matrices can be found with symbolic computation.

\section{A specialized Lyapunov theorem}
\label{sec:appendix2}
Here we present a Lyapunov function lemma on stability of $w(t)\in\mathbb{R}^m$ whose ODE is linear in $w$ but depends generally on $x(t)$ whose evolution can be nonlinear. This lemma is used to prove \cref{thm:shiftedmethod,thm:shifted dimension}. The second part, whose proof follows similar proofs in \citet{khalil2002nonlinear}, is a so-called converse Lyapunov theorem that guarantees existence of a Lyapunov function under a sufficiently strong stability assumption.
\begin{lem}
\label{thm:lyapunovfunction}
Let $x(t)$ and $w(t)$ solve the ODE system
\begin{equation}
\label{eq:generic augmented}
    \ddt{}\begin{bmatrix}
        x\\w
    \end{bmatrix} = \begin{bmatrix}
        f(x)\\ A(x)w
    \end{bmatrix}
\end{equation}
on $\mathcal B\times\mathbb{R}^m$ with initial condition $(x_0,w_0)$, where $f\in C^0(\mathcal{B},\mathbb{R}^n)$ and $A\in C^0(\mathcal{B}, \mathbb{R}^{m\times m})$. Assume $\mathcal{B}$ is forward invariant for the $x$ dynamics.
\begin{enumerate}
    \item Supopose each $x(t)$ is bounded for $t\ge0$. If there exists $V\in C^1\left(\mathcal{B}\times\mathbb{R}^m\right)$ such that
\begin{equation}
\label{eq:Lyapunov function conditions}
    V(x,w)\geq |w|^2\quad\text{ and }\quad f(x)\cdot D_xV(x,w) + \left[ A(x) w\right] \cdot D_wV(x,w)\leq 0
\end{equation}
for all $(x,w)\in\mathcal{B}\times\mathbb{R}^m$, then each $w(t)$ is bounded for $t\ge0$.
\item Suppose $\mathcal{B}$ is compact, $f\in C^1$ and $A\in C^1$. If there exist $K>0$ and $\lambda>0$ (uniform in $x_0$) such that
\begin{equation}
\label{eq:exp decay}
    |w(t)| \leq K e^{-\lambda t} |w_0|
\end{equation}
for all solutions to \cref{eq:generic augmented}, then there exists $U\in C^1\left(\mathcal{B},\mathbb{R}^{m\times m}\right)$ such that~\cref{eq:Lyapunov function conditions} is satisfied with $V(x,w) = w^{\mathsf{T}} U(x) w$.
\end{enumerate}
\end{lem}
\begin{proof}
For the first part, let $V\in C^1$ satisfy~\cref{eq:Lyapunov function conditions}. For contradiction, suppose that there is a solution to~\cref{eq:generic augmented} with $|w(t)|$ unbounded. The first condition on $V$ then implies that $V(x(t),w(t))$ is unbounded forward in time. The second condition on $V$ implies that
\begin{equation}
\ddt{}V(x(t),w(t)) = f(x(t))\cdot D_xV(x(t),w(t)) +\left[ A(x(t)) w(t) \right] \cdot D_wV(x(t),w(t)) \leq 0.
\end{equation}
This expression is continuous in $t$, so we can integrate to show that $V$ remains bounded for all time, which is a contradiction.

For the second part, fix a time $T > 0$, to be chosen later. Define
\begin{equation}
\label{eq:V lyap pf}
V(x_0, w_0) = \frac{2L}{1-e^{-2LT}}\int_0^T |w(t)|^2 \, {\rm d}t,
\end{equation}
where \( (x(t), w(t)) \) evolves under \eqref{eq:generic augmented} from initial condition $(x_0, w_0) $, and where $L$ is any value large enough that 
\begin{equation}
    \max_{x\in\mathcal{B}}\Vert A(x)\Vert \le L,
\end{equation}
where $\Vert\cdot\Vert$ denotes the $\ell^2$ operator norm, which is equal to the largest singular value of $A(x)$.  This condition on $L$ is needed below to ensure $V(x_0,w_0)\ge |w_0|^2$. By linearity of $w(t)$ with respect to $w_0$, the expression~\cref{eq:V lyap pf} for $V$ is quadratic in $w_0$ and thus can be written
\begin{equation}
    V(x_0,w_0) = w_0^\mathsf{T} U(x_0) w_0
\end{equation}
for some symmetric matrix $U(x_0)$. Since $A\in C^1(\mathcal{B}, \mathbb{R}^{m\times m})$, it follows that $U\in C^1(\mathcal{B}, \mathbb{R}^{m\times m})$.

Note that 
\begin{equation}
    \begin{aligned}
        \left|\ddt{} |w|^2\right| &\leq 2 |w| \left|A(x)w\right| \\
        &\leq2 \|A(x)\||w|^2 \\
        &\leq 2L|w|^2,
    \end{aligned}
\end{equation}
so $\ddt{} |w|^2 \geq -2L |w|^2$. Gr\"onwall's inequality then gives
\begin{equation}
    |w(t)|^2 \geq e^{-2Lt}|w_0|^2,
\end{equation}
and so
\begin{equation}
    V(x_0,w_0) \geq \frac{2L}{1-e^{-2LT}}\int_0^T e^{-2Lt} |w_0|^2 {\rm d}t = |w_0|^2,
\end{equation}
meaning that $V$ satisfies the first condition of \cref{eq:Lyapunov function conditions}.

For the second condition of \cref{eq:Lyapunov function conditions}, we define 
\begin{equation}
    v(t) = V(x(t),w(t)) = \frac{2L}{1-e^{-2LT}}\int_t^{t+T} |w(s)|^2 \, {\rm d}s.
\end{equation}
Then,
\begin{equation}
\begin{split}
     f(x(t))&\cdot D_xV(x(t),w(t)) + \left[ A(x(t)) w(t) \right] \cdot D_wV(x(t),w(t)) \\&= \ddt{v} \\&=  \frac{2L}{1-e^{-2LT}}\left(|w(t+T)|^2 - |w(t)|^2\right)
     \\&\leq  \frac{2L}{1-e^{-2LT}}\left(K^2 e^{-2\lambda T} -1\right) |w(t)|^2.
     \end{split}
\end{equation}
Choosing any $T\geq \frac{1}{\lambda} \log{K}$ ensures that the last line above is nonpositive, and so the second condition of \cref{eq:Lyapunov function conditions} is satisfied.
\end{proof}
\bibliographystyle{abbrvnat_nourl}
\bibliography{references.bib}

\end{document}